\documentclass[11pt]{amsart}

\usepackage{amsmath, amssymb, amsthm}

\usepackage{graphicx}

\usepackage[cmtip,all]{xy}

\usepackage{hyperref} 

\newtheorem{theorem}{Theorem}[section]
\newtheorem{corollary}[theorem]{Corollary}
\newtheorem{lemma}[theorem]{Lemma}
\newtheorem{proposition}[theorem]{Proposition}

\theoremstyle{definition}
\newtheorem{construction}[theorem]{Construction}
\newtheorem{definition}[theorem]{Definition}
\newtheorem{example}[theorem]{Example}

\newtheorem{question}[theorem]{Question}

\theoremstyle{remark}
\newtheorem{remark}[theorem]{Remark}

\DeclareMathOperator{\gldim}{gldim}

\DeclareMathOperator{\Ext}{Ext}

\newcommand{\id}{\operatorname{id}}
\newcommand{\Hom}{\operatorname{Hom}}

\newcommand{\im}{\operatorname{im}}

\newcommand{\Z}{\mathbb{Z}}

\newcommand{\la}{\langle}
\newcommand{\ra}{\rangle}
\newcommand{\tensor}{\otimes}
\newcommand{\tsr}{\tensor}

\renewcommand{\k}{\mathbb K}

\newcommand{\s}{\sigma}
\renewcommand{\l}{\lambda}

\newcommand{\g}{\gamma}

\newcommand{\p}{\rho}
\renewcommand{\t}{\tau}
\newcommand{\w}{\omega}

\newcommand{\N}{{\mathbb{N}}}
\renewcommand{\P}{{\mathbb{P}}}

\newcommand{\mc}{\mathcal}

\numberwithin{equation}{section}

\begin{document}

\title[Noncommutative Matrix Factorizations]{Curved DG Modules and Matrix Factorizations from  Noncommutative Quadric Hypersurfaces}


\author{Peter Goetz}
\address{Department of Mathematics and Data Science,
California State Polytechnic University, Humboldt,
Arcata, CA 95521}
\email{pdg11@humboldt.edu}

\date{}

\begin{abstract}
The category of noncommutative quadratic quadric hypersurfaces, ${\tt Quad}\text{-}{\tt QHS}$, consists of pairs $(A, f)$, where $A$ is a quadratic algebra and $f \in A$ is a nonzero degree $2$ element. We associate to such $(A, f)$ a pair $(\bar{A}^!, f^!)$, and show that this association makes ${\tt Quad}\text{-}{\tt QHS}$ into a category with duality. We construct a faithful functor from the category of graded modules over $\bar{A}^!$ to the homotopy category of curved DG modules over a canonical curved DG algebra $(A \tsr \bar{A}^!, d, f \tsr f^!)$. If $A$ satisfies the left strong rank condition and $f \in A$ is not a right zero divisor, we show that the restriction of our functor to a natural full subcategory of the category of graded modules over $\bar{A}^!$ is valued in a stable category of noncommutative matrix factorizations of $f$. When $A$ is Koszul of finite global dimension and $f \in A$ is normal and regular, we prove that the even Clifford algebra, $\bar{A}^![(f^!)^{-1}]_0$, is isomorphic to a canonical PBW-deformation of a Zhang twist of the $2$-Veronese subalgebra of the Koszul dual $A^!$. Finally, we study several classes of Artin-Schelter regular algebras to illustrate our results. 
\end{abstract}

\maketitle


\section{Introduction} 

\subsection{Background}

The notion of a \emph{matrix factorization} of an element in a commutative ring was introduced by Eisenbud, \cite{Eisenbud}. If $A$ is a commutative ring and $f \in A$, a matrix factorization of $f$ is a pair of maps $\phi: F \to G$, $\psi: G \to F$ between free $A$-modules such that $\psi \phi = f \id_F$ and $\phi \psi = f \id_G$. In the case when $A$ is a regular local ring and $f \in A$, \cite[Theorem 6.1]{Eisenbud} shows that matrix factorizations are intimately related to periodic resolutions and maximal Cohen-Macaulay (MCM) modules over the hypersurface ring $A/(f)$. 

Let $\k$ be a field. In \cite{BEH}, Buchweitz-Eisenbud-Herzog considered the case of a commutative polynomial ring $A = \k[x_1, \ldots, x_n]$ and a quadratic form $f \in A$ . They proved, \cite[Theorem 2.1]{BEH}, there is an equivalence of categories from the category of graded MCM $A/(f)$-modules without free summands to the category of modules over the \emph{even Clifford algebra} $C_0(f)$. More generally, in \cite[Appendix]{BEH} Buchweitz explained that this equivalence is closely related to the BGG-correspondence associated to a pair of graded Gorenstein noetherian Koszul dual algebras $(A, A^!)$, also see \cite[Theorem 4.4.1, Theorem C.1.3]{Buchweitz}.

Many researchers have built upon the works \cite{Buchweitz, BEH, Eisenbud} and extended these ideas to noncommutative settings. Of relevance to this paper, we were motivated by \cite{CCKM, He-Ye, Mori-Ueyama-1, Mori-Ueyama-2, Smith-vdB, U}. We summarize some of the main results in these works as follows. Let $A$ be a connected graded noetherian $\k$-algebra, let ${\tt GrMod}\text{-}A$ and ${\tt Tors}\text{-}A$ denote the category of graded right $A$-modules and the category of graded torsion right $A$-modules, respectively. Then the \emph{noncommutative projective scheme} associated to $A$ is by definition the Serre quotient \[{\tt Tails} \, A := \dfrac{{\tt GrMod}\text{-}A}{{\tt Tors}\text{-}A}.\]  The \emph{global dimension} of ${\tt Tails} \, A$ is defined as
\[\gldim({\tt Tails} \, A) := \sup\big\{i : \Ext^i_{{\tt Tails} \, A}(\mathcal{M},\mathcal{N}) \ne 0 \text{ for some } \mathcal{M}, \mathcal{N} \in {\tt Tails} \, A\big\}.\] The category ${\tt Tails} \, A$ is \emph{smooth} if it has finite global dimension. 

Smith-Van den Bergh studied the case where $A$ is a connected graded noetherian Gorenstein Koszul algebra of finite global dimension and $f \in A_2$ is a \emph{central regular} element. Letting $\bar{A} = A/\la f \ra$,  \cite[Proposition 5.2]{Smith-vdB} states an equivalence from a stable category of graded MCM modules over $\bar{A}$ to the bounded derived category of ${\tt Tails} \, \bar{A}^!$. Furthermore, there is an equivalence ${\rm D}^{\rm b}({\tt Tails} \, \bar{A}^!) \simeq {\rm D}^{\rm b}({\tt mod}\text{-}C(A))$, where $C(A)$ is a finite-dimensional algebra analogous to the even Clifford algebra, and moreover, if $C(A)$ is semisimple, then ${\tt Tails} \, A$ is smooth. 

These results were extended by Mori-Ueyama to the case where $f \in A$ is only required to be a \emph{normal regular} element. In \cite{Mori-Ueyama-1} they defined a category of \emph{graded noncommutative matrix factorizations} and showed that if $A$ is a noetherian Artin-Schelter regular algebra and $f \in A_d$ is a normal regular element, then there is an equivalence from a stable category of noncommutative matrix factorizations of $f$ over $A$ to a stable category of graded MCM modules over $\bar{A}$, \cite[Theorem 6.6]{Mori-Ueyama-1}. In subsequent work, \cite{Mori-Ueyama-2}, they defined the even Clifford algebra $C(A)$ when $f \in A_2$ is normal and regular, and proved an equivalence from the stable category of graded MCM modules over $\bar{A}$ to the bounded derived category of finitely generated modules over $C(A)$, \cite[Lemma 4.13]{Mori-Ueyama-2}. Furthermore, they proved that the algebra $C(A)$ is semisimple if and only if the category ${\tt Tails}\, \bar{A}$ is smooth. To give a more precise statement of this result, recall that Ueyama, \cite[Definition 2.2]{U}, defines an algebra $A$ to be a \emph{graded isolated singularity} if ${\tt Tails} \, A$ is smooth. 

\begin{theorem}[{\cite[Theorem 5.5]{Mori-Ueyama-2}}]\label{isolated singularity}
Let $A$ be a quadratic noetherian Artin-Schelter regular algebra of global dimension $n$ and let $f \in A_2$ be a normal regular element. Let $\bar{A} = A/\la f \ra$ and let $C(A) = \bar{A}^![(f^!)^{-1}]_0$. Then the following statements are equivalent.
\begin{itemize}
\item[(1)] $\bar{A}$ is a graded isolated singularity. 
\item[(2)] $\bar{A}$ has finite Cohen-Macaulay representation type.
\item[(3)] $C(A)$ is semisimple. 
\end{itemize}
\end{theorem}

He-Ye, \cite{He-Ye}, studied the even Clifford algebra $C(A)$ in the case where $A$ is a noetherian Koszul Artin-Schelter regular algebra and $f \in A_2$ is a \emph{central} regular element. They proved that $C(A)$ is isomorphic to the degree $0$ part of a certain \emph{Clifford deformation} of the Koszul dual algebra $A^!$. 

Before discussing the main contributions of this paper, we draw attention to two observations regarding the above results. First, notice that the domains of all of the above equivalence functors are either stable categories of MCMs or categories of noncommmutative matrix factorizations. If one desires to actually construct MCMs or noncommutative matrix factorizations one could, in theory, write down quasi-inverses of these functors, however, this does not seem to be so easy to do. Second, most of the above results, at a minimum, require $A$ to be a Koszul algebra as input. 

In this paper, we work in greater generality and consider any quadratic algebra $A$ and any nonzero element $f \in A_2$. We refer to such a pair $(A, f)$ as a \emph{noncommutative quadratic quadric hypersurfaces}. We associate to such $(A, f)$ a canonical curved differential graded algebra $T^{\bullet}$. One of our main results, see Theorem \ref{Section 4 main} below, constructs a canonical faithful functor from the graded module category over $\bar{A}^!$ to the homotopy category of curved differential graded modules over the curved differential graded algebra $T^{\bullet}$. We show that the restriction of our functor to a natural full subcategory produces graded noncommutative matrix factorizations of $f$ over $A$ in the sense of \cite{Mori-Ueyama-1}. We also generalize the above mentioned result in \cite{He-Ye} regarding the even Clifford algebra. We show that when $A$ is a noetherian Koszul algebra of finite global dimension and $f \in A_2$ is normal and regular, then the even Clifford algebra is isomorphic to a canonical PBW-deformation of a Zhang twist of the $2$-Veronese subalgebra of the quadratic dual algebra $A^!$, see Theorem \ref{Cl(A, f) as PBW deformation}.

\subsection{Outline and main results}

In Section 2 we state preliminary results and establish notation. Most of Section 2 is standard and there are no new results, with the exception of, perhaps, Proposition \ref{easy CDGs} and Proposition \ref{easy CDG modules} where we give a general construction of some curved differential graded algebras and curved modules over them.

In Section 3 we consider quadratic or Koszul algebras $A$ along with quadratic elements
$f \in A_2$. We call such pairs $(A, f)$ \emph{noncommutative quadratic or Koszul quadric hypersurfaces}. If $A^!$ denotes the quadratic dual of $A$, then it is well known that there is a canonical quadratic element $f^! \in (A/\la f \ra)^!$, unique up to a scalar, such that $(A/ \la f \ra)^! \cong A^!$. We make the correspondence $(A, f) \to (A^!, f^!)$ into a functor by defining categories of noncommutative quadratic or Koszul quadric hypersurfaces. The main result of Section 3 is the following. The reader may refer to Definition \ref{category NCQuadQHypSurf} and Definition \ref{category NCKosHypSurf} for the unexplained notation.

\begin{theorem}[Theorem \ref{NCQuadQHypSurf is category with duality}, Corollary \ref{NCKosHypSurf is category with duality}]\label{Section 3 main}
The following triples are categories with duality
\begin{itemize}
\item[(1)] $({\tt Quad\text{-}QHS}, ( \ )^!, \eta)$,
\item[(2)] $(\tt{Kos\text{-}QHS}, ( \ )^!, \eta)$.
\end{itemize}

\end{theorem}

In Section 4 we associate to any noncommutative quadratic quadric hypersurface $(A, f)$ a canonical curved differential graded algebra $(T, d_T, \theta_T)$. Furthermore, under some additional conditions on the pair $(A, f)$, this association is strongly related to the notion of graded left matrix factorizations of $f$ in the sense of \cite{Mori-Ueyama-1}. The main result of Section 4 is the following. For some of the unexplained notation, the reader may refer to Proposition \ref{T^bullet CDG} and Definition \ref{category I}.

\begin{theorem}[Theorem \ref{the functor F}, Theorem \ref{F functor valued in NMF}, Proposition \ref{F functor valued in stable NMF}]\label{Section 4 main} 

Suppose that $(A, f) \in {\tt Quad}\text{-}{\tt QHS}$ and consider the curved DG algebra $T^{\bullet} = (A \tsr \bar{A}^!, d_T, f \tsr f^!)$.

\begin{itemize}
\item[(1)] There is a faithful functor ${\mc F}: {\tt GrMod}\text{-}\bar{A}^! \to {\tt Ho}({\tt CDGMod}\text{-}T^{\bullet})$.
\item[(2)] If $A$ satisfies the left strong rank condition and $\p_f: A \to A$ is injective, then there is a faithful functor ${\mc F}: {\tt P}(\bar{A}^!, f^!) \to {\underline {\tt NMF}}_{\Z}^A(f)$. 

\end{itemize}

\end{theorem}

In Section 5 we restrict attention to noetherian Koszul algebras $A$ of finite global dimension and normal regular elements $f \in A_2$. By \cite[Lemma 4.2, Proposition 4.6]{Mori-Ueyama-2} the noncommutative projective scheme ${\tt Tails} \, \bar{A}^!$ is equivalent to the category of right modules over the \emph{even Clifford algebra} ${\mc C}(A, f) = \bar{A}^![(f^!)^{-1}]_0$. Motivated by this, we have the following main result of Section 5. 

\begin{theorem}[Theorem \ref{NMFs from reps}, Theorem \ref{Cl(A, f) as PBW deformation}]\label{Section 5 main} 

Let $A$ be a noetherian Koszul algebra of finite global dimension, and let $(A, f) \in {\tt Kos}\text{-}{\tt QHS}$. 

\begin{itemize}
\item[(1)] There exists a faithful functor ${\mc F}:{\tt fdim}\text{-} \mc{C}(A, f) \to \underline{{\tt NMF}}_{\Z}^A(f)$. 
\item[(2)] The even Clifford algebra ${\mc C}(A, f)$ is isomorphic to a PBW-deformation of the Zhang twist algebra $[A^!_{(2)}]^{\Phi}$.
\end{itemize}

\end{theorem}

Finally, in Section 6 we apply our main results by studying three examples of Artin-Schelter regular algebras of global dimensions $2$, $3$ and $4$, and some of their quadratic normal regular elements. Using our constructions, we find interesting noncommutative matrix factorizations of these elements.

\section{Preliminaries}

In this section we gather together various standard facts we will need, as well as fixing notation. For the reader's convenience, we have divided the section into subsections with descriptive titles. Aside from perhaps Proposition \ref{easy CDGs} and Proposition \ref{easy CDG modules}, nothing in this section is new, so the reader may want to skip this section and refer back to it as necessary. Let $\k$ be a field.

\subsection{Vector spaces}\label{vector spaces}

Let ${\tt Vect}$ be the category of $\k$-vector spaces and $\k$-linear maps. The term \emph{vector space} refers to an object of ${\tt Vect}$. The tensor product of vector spaces $V$ and $W$ is denoted by $V\tsr W$ and $\Hom(V, W)$ is the vector space of $\k$-linear maps from $V$ to $W$. We write $V^* = \Hom(V, \k)$ for the $\k$-linear dual of the vector space $V$. If $V$ is a $\Z$-graded vector space and $v \in V$ is homogeneous, we write $|v|$ for the degree of $v$. When $V$ and $W$ are $\Z$-graded, 
the space $V\tsr W$ admits a $\Z$-grading via the formula
\[(V\tsr W)_m=\bigoplus_{k+\ell=m} V_k\tsr W_{\ell}.\]  

If $V$ is a finite-dimensional vector space and $n$ is a positive integer, consider the natural pairing between $V^{\tsr n}$ and $(V^*)^{\tsr n}$, \[\la \ , \ \ra: V^{\tsr n} \tsr (V^*)^{\tsr n} \to \k\] 
given by \[\la v_1 \tsr \cdots \tsr v_n, \l_1 \tsr \cdots \tsr \l_n\ra = \l_1(v_1) \cdots \l_n(v_n),\] 
where $\l_i \in V^*$ and $v_i \in V$ for all $i$. We will also denote the right hand side of the last equation either by $[\l_1 \tsr \cdots \tsr \l_n](v_1 \tsr \cdots \tsr v_n)$ or by $[v_1 \tsr \cdots \tsr v_n](\l_1 \tsr \cdots \tsr \l_n)$. If $I$ is a subspace of $V^{\tsr n}$ or $(V^*)^{\tsr n}$, then $I^{\perp}$ will denote the \emph{orthogonal complement} of $I$ in $(V^*)^{\tsr n}$ or $V^{\tsr n}$, respectively, under this pairing. 

Let $U, V, W \in {\tt Vect}$, and assume that $V$ and $W$ are finite-dimensional. It is well known that the map \[\Omega: V \tsr W \to \Hom(V^*, W),  \quad \big[\Omega\big(\sum_i v_i \tsr w_i\big)\big](\l) = \sum_i \l(v_i)w_i\] is an isomorphism in ${\tt Vect}$. Let \[\Omega_{*}: \Hom(U, V \tsr W) \to \Hom(U, \Hom(V^*, W))\] 
be the pushforward isomorphism associated to $\Omega$.  Let \[\Psi: \Hom(U \tsr V^*, W) \to \Hom(U, \Hom(V^*, W))\] 
be the canonical adjoint isomorphism. Let $\Delta = \Omega_{*}^{-1} \circ \Psi$, then we have the commutative diagram 
\[\xymatrix{&\Hom(U \tsr V^*, W) \ar[rr]^-{\Psi} \ar@/_1pc/[rrd]_-{\Delta} & &  \Hom(U, \Hom(V^*, W)) \\ & &  &\Hom(U, V \tsr W). \ar[u]^-{\Omega_{*}}
}
\]

We will need the following explicit description of the map $\Delta$. Let $\{x_k\} \subset V$ and $\{\lambda_k\} \subset V^*$ be a pair of dual bases. For all $\p \in \Hom(U \tsr V^*, W)$, the map $\Delta(\p)$ is given by 

\begin{equation}
\label{Delta}
 \Delta(\p)(u) = \sum_{k} x_k \tsr \p(u \tsr \l_k) \quad \text{for all } u \in U.
\end{equation}

To see this, note that it suffices to check that \[\Psi(\p)(u) = \Omega\big(\sum_k x_k \tsr \p(u \tsr \l_k)\big) \quad \text{for all } u \in U,\] 
and this follows by evaluating each side on the $\l_i$ and unwinding definitions. We will refer to the map $\Delta(\p)$ as the \emph{adjoint of $\p$}. 

\subsection{Associative algebras}\label{algebras}

In this paper, the term \emph{algebra} refers to an associative $\k$-algebra, and \emph{graded algebra} refers to a $\Z$-graded, locally finite-dimensional, associative $\k$-algebra. A \emph{connected graded algebra} is an $\N$-graded algebra $A$ such that $A_0 = \k$. The \emph{tensor algebra} on the vector space $V$ is denoted by $T(V)$. We always consider $T(V)$ as an $\N$-graded algebra in the usual way; that is, $V$ is considered as the degree $1$ component of $T(V)$. If $A$ is a graded algebra and $d \in \N$, the \emph{$d$-Veronese} of $A$ is the subalgebra \[A_{(d)} = \bigoplus_{i \in \Z} A_{di}\] of $A$. Let ${\tt GrAlg}$ denote the category of graded algebras with degree-preserving algebra morphisms. We say that a graded algebra $A$ is \emph{one-generated} if the natural map $\pi_A: T(A_1) \to A$ is surjective. Let ${\tt GrAlg}^1$ denote the full subcategory of ${\tt GrAlg}$ consisting of all one-generated algebras. For any morphism $\varphi:A\to B$ in ${\tt GrAlg}^1$, there is the induced morphism $\varphi: T(A_1)\to T(B_1)$, also in ${\tt GrAlg}^1$. Moreover, it is clear that $\varphi(\ker \pi_A)\subset \ker \pi_B$. In fact, the association $A \to T(A_1)$ yields a functor $T: {\tt GrAlg}^1 \to {\tt GrAlg}^1$. For the purposes of this paper, there is no harm or loss in generality in assuming that the objects in ${\tt GrAlg}^1$ are explicitly presented in the form $T(V)/I$ for some finite-dimensional vector space $V$ and some ideal $I$ of $T(V)$. Furthermore, a morphism $T(V)/I \to T(W)/J$ in ${\tt GrAlg}^1$ determines, and is determined by, a morphism $T(V) \to T(W)$  in ${\tt GrAlg}^1$ carrying $I$ into $J$.  An algebra $T(V)/I \in {\tt GrAlg}^1$ is a \emph{quadratic algebra} if the ideal $I \subset T(V)$ is generated by its degree $2$ component $I_2$. We let ${\tt Quad}$ denote the full subcategory of ${\tt GrAlg}^1$ consisting of all quadratic algebras. We will denote objects of ${\tt Quad}$ as $A = T(V)/\la I_A \ra$, with the implication that $I_A$ is a subspace of $V \tsr V$. 

If $A$ is a graded algebra and $f \in A$, then $f$ is \emph{normal} if $fA = Af$. The element $f \in A$ is \emph{regular} if $f$ is neither a left nor a right zero-divisor in $A$. We will say that $f \in A$ is \emph{1-regular} if both left and right multiplication by $f$ on $A_1$ is injective. If $f$ is both normal and regular, then the \emph{normalizing automorphism of $f$} is the algebra isomorphism $\varphi_f: A \to A$ determined by $fa = \varphi_f(a)f$ for all $a \in A$. When $A$ is one-generated and $f \in A$ is normal, regular and homogeneous, then the normalizing automorphism of $f$ is graded and is completely determined by its restriction to $A_1$.

Let $A = \bigoplus_{i \in \Z} A_i$ be a $\Z$-graded associative $\k$-algebra. For any graded algebra automorphism $\zeta$ of $A$, the \emph{Zhang twist} of the algebra $A$ by the automorphism $\zeta$ is the $\Z$-graded associative $\k$-algebra $A^{\zeta}$ given by $A^{\zeta} =  \bigoplus_{i \in \Z} A_i$ as $\Z$-graded vector spaces with multiplication given by \[x \ast y = x \zeta^n(y) \quad \text{for all } \,  x \in A_n, \, y \in A.\]
We refer to \cite{Zhang} for more details.  

\subsection{Modules}\label{modules}

If $A$ is an algebra, then $A\text{-}\tt{Mod}$ denotes the category of left $A$-modules. The categories $A\text{-}\tt{mod}$ and $A\text{-}\tt{fdim}$ are the full subcategories of $A\text{-}\tt{Mod}$ consisting of finitely generated modules, respectively, finite-dimensional modules. We use corresponding notation for right $A$-module categories: ${\tt{Mod}}\text{-}A$, ${\tt{mod}}\text{-}A$ and ${\tt{fdim}}\text{-}A$. If $M \in A\text{-}\tt{Mod}$ and $a \in A$, then $\l_a: M \to M$ denotes left multiplication by $a$. Similarly, if $M$ is a right $A$-module, then $\p_a: M \to M$ denotes right multiplication by $a \in A$. If $A$ is noetherian (meaning left and right noetherian), then the \emph{global dimension of $A$} is denoted by $\gldim A$. Recall that $A$ is said to satisfy the \emph{left strong rank condition} if whenever $A^m \to A^n$ is a monomorphism of finitely generated free left $A$-modules, then $m \leq n$; see \cite[Section 1.1]{Lam}. It is well known that every left noetherian algebra satisfies the left strong rank condition. 

If $A$ is a graded algebra, we denote by $A \text{-} {\tt GrMod}$ the category of $\Z$-graded left $A$-modules, with degree-preserving morphisms. We have the following full subcategories of $A \text{-} {\tt GrMod}$: $A\text{-}\tt{grmod}$ consists of all finitely generated graded modules, and $A\text{-}\tt{grfdim}$ consists of all finite-dimensional graded modules. For categories of graded right $A$-modules we use corresponding notation: ${\tt{GrMod}}\text{-}A$, ${\tt{grmod}}\text{-}A$ and ${\tt{grfdim}}\text{-}A$.

For $M\in A\text{-}{\tt GrMod}$ and $i\in \Z$, we denote by $M(i)$ the degree-shifted graded left $A$-module defined by $M(i)_n=M_{n+i}$ for all $n\in\Z$. The \emph{graded Hom functor} for graded left $A$-modules is given by
\[\Hom_A(M,N)=\bigoplus_{i\in\Z} \Hom_{A\text{-}{\tt GrMod}}(M,N(i)).\]
 This functor is left exact, and its $i$-th right derived functor is denoted $\Ext^i_A(M,N)$. The space $\Ext^i_A(M,N)$ inherits a grading from the graded Hom functor, and we denote the homogeneous degree $j$ component by $\Ext^{i,j}_A(M,N)$. One may compute the space $\Ext^i_A(M,N)$ as the homology group $H^i(P_{\bullet},N)$ where $(P_{\bullet},d_{\bullet})$ is a graded free resolution of $M$. Thus 
\[\Ext_A(M,N)=\bigoplus_i \Ext^i_A(M,N)=\bigoplus_{i,j} \Ext^{i,j}_A(M,N)\] is bigraded.
When $A$ is a connected graded algebra and $M=N= \k$ is the trivial module, the vector space $E(A)=\Ext_A(\k,\k)$ admits the structure of a bigraded algebra via the Yoneda composition product, see \cite[p. 4]{PP} for more details. 
The algebra $E(A)$ is called the \emph{Yoneda algebra} of $A$.

\subsection{Quadratic and Koszul algebras}\label{quadratic and Koszul}

Let $A = T(V)/\la I_A \ra$ be a quadratic algebra. Then $I_A^{\perp}$, the orthogonal complement of $I_A$, is a subspace of $V^* \tsr V^*$. The \emph{quadratic dual} of $A$ is the quadratic algebra $A^! := T(V^*)/\la I_A^{\perp} \ra$. If $A = T(V)/\la I_{A} \ra$ and $B = T(W)/\la I_{B} \ra$, and $\varphi: A \to B$ is a morphism in ${\tt Quad}$, then there is a natural morphism $\varphi^{!}: B^! \to A^!$ in ${\tt Quad}$ induced by the $\k$-linear dual map $\varphi^*: W^* \to V^*$, and of course $\varphi^* : T(W^*) \to T(V^*)$ carries $\la (I_B)^{\perp} \ra$ into $\la (I_A)^{\perp} \ra$.

Recall that a connected graded algebra $A$ is \emph{Koszul} if the trivial module $\k$ admits a resolution in $A\text{-}{\tt GrMod}$
\[\cdots\to P_3\to P_2\to P_1\to P_0 \to \k\to 0\]
such that each $P_i$ is a graded free left $A$-module generated in degree $i$. It is immediate from the definition that every Koszul algebra is quadratic.  
It is well known that $A$ is Koszul if and only if the canonical homomorphism $A^! \to E(A)$ is an algebra isomorphism. 

\subsection{Curved algebras and modules}

In this section we discuss the notions of curved differential graded algebras and modules. Following \cite{PP}, a \emph{curved differential graded algebra} (\emph{CDG-algebra}) is a triple $B^{\bullet} = (B, d_B, \theta_B)$ consisting of a $\Z$-graded, associative $\k$-algebra $B = \oplus_{i \in \Z} B_i$, an odd derivation of degree $1$ \[d_B: B \to B, \quad d_B(xy) = d_B(x)y + (-1)^{|x|}x d_B(y) \quad \text{for all homogeneous } x, y \in B,\] and an element \[\theta_B \in B_2 \ \ \text{such that} \ \ d_B(\theta_B) = 0 \ \ \text{and} \ \ d_B^2(x) = [\theta_B, x] = \theta_B x - x \theta_B \ \ \text{for all }  x \in B.\] The element $\theta_B$ is called the \emph{curvature element}. There is a simple way to produce examples of CDG-algebras.

\begin{proposition}\label{easy CDGs}
Let $B = \oplus_{i \in \Z} B_i$ be a $\Z$-graded, associative $\k$-algebra and let $b \in B_1$. Define $d_b: B \to B$ by $d_b(x) = bx - (-1)^{|x|}xb$ for all homogeneous $x \in B$.  Then the triple $(B, d_b, b^2)$ is a CDG-algebra.
\end{proposition}

\begin{proof}
Let $c = b^2$. It is routine to check that $d_b$ is an odd derivation of degree $1$, and clearly $d_b(c) = 0$. We will check that $d_b^2(x) = [c, x]$. First note that $d_b(b) = 2b^2 = 2c$. For homogeneous $x \in B$ we have 
\begin{align*}
d_b^2(x) &= d_b(bx-(-1)^{|x|}xb) \\
&= d_b(bx) - (-1)^{|x|} d_b(xb) \\
&= d_b(b)x-bd_b(x)-(-1)^{|x|}(d_b(x)b + (-1)^{|x|}d_b(b)) \\
&= 2cx -b(bx-(-1)^{|x|}xb) - (-1)^{|x|}(bx-(-1)^{|x|}xb)b-2xc \\
&= [c, x].
\end{align*}
\end{proof}

A \emph{morphism of CDG-algebras} is $(f, c): (B, d_B, \theta_B) \to (C, d_C, \theta_C)$ such that $f: B \to C$ is a graded algebra morphism and $c \in C_1$, satisfying \begin{align*} f(d_B(x)) &= d_C(f(x)) + \{c, f(x)\} \\ 
 f(\theta_B) &= \theta_C+d_C(c) + c^2, \end{align*} where $\{x, y\} = xy-(-1)^{|x||y|} yx$. We note that the construction in Proposition \ref{easy CDGs} can be made into a functor. If $f: B \to C$ is a graded algebra morphism such that $f(b) = c$ for $b \in B_1$, then $(f, 0): (B, d_b, b^2) \to (C, d_c, c^2)$ is easily checked to be a morphism of CDG-algebras. It is clear that this correspondence between pairs $(B, b)$ and the CDG-algebra $(B, d_b, b^2)$ determines a faithful covariant functor. 
 
 Now we define curved modules over a CDG-algebra. A \emph{right curved differential graded module} (\emph{right CDG-module}) over the CDG-algebra $(B, d_B, \theta_B)$ is a pair $(N, d_N)$ consisting of a $\Z$-graded right $B$-module $N$ and an odd derivation of degree $1$, $d_N: N \to N$: \[d_N(nx) = d_N(n)x+(-1)^{|n|}nd_B(x) \quad \text{for all homogeneous } n \in N, x \in B,\] such that $d^2_N(n) = -n \theta_B$ for all $n \in N$. 
 
 The next result constructs CDG-modules over the CDG-algebras defined in Proposition \ref{easy CDGs}. The proof is straightforward and left to the reader. 
 
 \begin{proposition}\label{easy CDG modules}
 Let $B$ be a $\Z$-graded, associative $\k$-algebra, let $b \in B_1$, and consider the CDG-algebra $(B, d_b, b^2)$ defined in Proposition \ref{easy CDGs}. Let $N \in {\tt GrMod}\text{-}B$, and define a map $d_{N, b}: N \to N$ by $d_{N, b}(n) = (-1)^{|n|+1}nb$ for all homogeneous $n \in N$. Then the pair $(N, d_{N, b})$ is a right CDG-module over $(B, d_b, b^2)$.
 \end{proposition}
 
 \begin{proof}
 This is straightforward. 
 \end{proof}
 
 Let $(B, d_B, \theta_B)$ be a CDG-algebra. We now describe the \emph{DG category of right CDG-modules over $B^{\bullet} = (B, d_B, \theta_B)$}. This DG category will be notated as ${\tt CDGMod}\text{-}B^{\bullet}$. A basic general reference for DG categories is \cite[Chapter 3]{Y}. Let $(N_1, d_{N_1})$ and $(N_2, d_{N_2})$ be objects  in ${\tt CDGMod}\text{-}B^{\bullet}$. Since $N_1$ and $N_2$ are $\Z$-graded right $B$-modules we may consider the $\Z$-graded vector space \[\Hom_{B}(N_1, N_2) = \bigoplus_{i \in \Z} \Hom_{{\tt GrMod}\text{-}B}(N_1, N_2(i)).\] Then it is well known that $\Hom_{B}(N_1, N_2)$ is a DG vector space with degree $1$ differential $d: \Hom_{B}(N_1, N_2)_i \to \Hom_{B}(N_1, N_2)_{i+1}$ given by \[[d(\phi)](n_1) = d_{N_2} \phi(n_1) - (-1)^{i} \phi d_{N_1}(n_1) \quad \text{for all } \phi \in \Hom_{B}(N_1, N_2)_i, \ n_1 \in N_1.\] One can easily check that $d^2 = 0$. 
 
 We will denote the \emph{homotopy category of right CDG modules over $(B, d_B, \theta_B)$} by ${\tt{Ho}}({\tt{CDGMod}}\text{-}B^{\bullet})$. This is the category with the same objects as ${\tt CDGMod}\text{-}B^{\bullet}$ with morphism sets equal to $H^0(\Hom_B(N_1, N_2))$. In more detail, morphisms in ${\tt{Ho}}({\tt{CDGMod}}\text{-}B^{\bullet})$ are cochain maps of degree $0$, $\phi: N_1 \to N_2$, up to coboundaries of $d$. We will also write 
 \begin{align*} &Z^{0}(\Hom_B(N_1, N_2)) = \ker(d\colon \Hom_B(N_1, N_2)_0 \to \Hom_B(N_1, N_2)_1), \\ 
 &B^{0}(\Hom_B(N_1, N_2)) = \im(d\colon \Hom_B(N_1, N_2)_{-1} \to \Hom_B(N_1, N_2)_0).\end{align*}
 See \cite[Definition 3.4.6]{Y} for more details.

\section{Categories of noncommutative quadratic and Koszul quadric hypersurfaces} 

Given a quadratic algebra $A$ and $f \in A_2$, it is well known that there is a canonical quadratic element $f^! \in (A/\la f \ra)^!$, unique up to an element of $\k^*$, such that $(A/\la f \ra)^!/\la f^! \ra \cong A^!$, see \cite{PP, Shelton-Tingey}. In this section we make this correspondence more precise as a functorial construction. To do this, we introduce two categories and then prove that they are categories with duality.

\begin{definition}\label{category NCQuadQHypSurf}
Let ${\tt Quad\text{-}QHS}$ denote the category whose objects are pairs $(A, f)$, where $A$ is a quadratic algebra and $0 \ne f \in A_2$. By definition, a morphism in ${\tt Quad\text{-}QHS}$, say $(A, f) \to (B,g)$, is a morphism $\varphi: A \to B$ in ${\tt GrAlg}$ such that $\varphi(f) = g$. We call ${\tt Quad\text{-}QHS}$ the \emph{category of noncommutative quadratic quadric hypersurfaces}. 
\end{definition}

\begin{construction}\label{duality functor on NC-Quad-QHypSurf}
We construct a contravariant functor \[( \ ) ^{!}: {\tt Quad\text{-}QHS} \to {\tt Quad\text{-}QHS}\] 
as follows. 
First, we define $( \ ) ^{!}: {\tt Quad\text{-}QHS} \to {\tt Quad\text{-}QHS}$ on objects. Consider some pair $(A, f) \in {\tt Quad\text{-}QHS}$, where $A = T(V)/\la I_A \ra$ and $0 \ne f \in A_2$. Choose any $F \in V \tsr V$ such that $f= F + I_A$, and consider the subspace $I_A + \k F$ of $V \tsr V$. Note that this sum is direct and the subspace does not depend on the choice of $F$. We let $\bar{A} = T(V)/\la I_A + \k F \ra$, and define $I_{\bar{A}} = I_A + \k F$. The quadratic dual of $\bar{A}$ is $\bar{A}^! = T(V^*)/\la (I_{\bar{A}})^{\perp} \ra$. It is not hard to check that the map defined by
\begin{equation}
\label{psi_(A,f)}
\psi_{(A, f)}: \bar{A}^!_2 \to I_{\bar{A}}^*, \quad [\psi_{(A, f)}(t+(I_{\bar{A}})^{\perp})](r) = t(r)
\end{equation}
is a vector space isomorphism. Define the linear functional $\phi_{(A, f)} \in I_{\bar{A}}^*$ by 
\begin{equation}
\label{phiAf}
\phi_{(A, f)}(I_A) = 0 \ \ \text{and} \ \ \phi_{(A,f)}(F) = 1,
\end{equation}
and note that $\phi_{A, f}$ does not depend on the choice of $F \in V \tsr V$ made above. Let \[f^! = \psi^{-1}_{(A,f)}(\phi_{(A, f)}) \in \bar{A}^!_2,\]
then one checks that $f^!$ only depends on $f$, so the notation $f^!$ is well-chosen. Noting that, indeed, $(\bar{A}^!, f^!) \in {\tt Quad\text{-}QHS}$, we define \[(A, f)^! = (\bar{A}^!, f^!).\]

Now we define $( \ ) ^{!}: {\tt Quad\text{-}QHS} \to {\tt Quad\text{-}QHS}$ on morphisms. Let $(A, f)$, $(B, g)$ be objects in ${\tt Quad\text{-}QHS}$ with $A = T(V)/\la I_A \ra$, $0 \ne f \in A_2$ and $B = T(W)/\la I_B \ra$, $0 \ne g \in B_2$. Then we may consider the pairs $(\bar{A}^!, f^!)$ and $(\bar{B}^!, g^!)$. Suppose that we have some morphism $\varphi: (A, f) \to (B, g)$ in ${\tt Quad\text{-}QHS}$. Since $\varphi: A \to B$ maps $f$ to $g$, it is clear that we have an induced morphism $\bar{\varphi}: \bar{A} \to \bar{B}$. So, taking quadratic duals, we have a morphism $\bar{\varphi}^!: \bar{B}^! \to \bar{A}^!$. Now we will check that $\bar{\varphi}^!$ maps $g^!$ to $f^!$. By assumption, $\varphi(f) = g$, so it follows that if we have chosen $F \in V \tsr V$ and $G \in W \tsr W$ such that $f = F+I_A$ and $g = G+I_B$, then we have $\varphi(F) - G \in I_B$. Consider the following diagram 
\[\xymatrix{&(\bar{B})^!_2 \ar[r]^-{\psi_{(B, g)}} \ar[d]_-{\bar{\varphi}^!} & (I_{\bar{B}})^* \ar[d]^{\rho} \\ & (\bar{A})^!_2 \ar[r]^-{\psi_{(A, f)}} & (I_{\bar{A}})^*.
}
\]
The rightmost vertical map, $\rho: (I_{\bar{B}})^* \to (I_{\bar{A}})^*$, is the $\k$-linear dual of the map $\overline{\varphi} \tsr \overline{\varphi}: I_{\bar{A}} \to I_{\bar{B}}$. A straightforward unwinding of the definitions shows that this diagram is commutative. It is easy to check that $\rho(\phi_{(B, g)}) = \phi_{(A, f)}$ as linear functionals on $I_{\bar{A}}$, so using commutativity of the diagram we have 
\begin{align*}
\varphi^!(g^!) &= \varphi^!(\psi^{-1}_{(B, g)}(\phi_{(B, g)})) \\
&= \psi^{-1}_{(A, f)} \rho(\phi_{(B, g)}) \\
&= \psi^{-1}_{(A, f)}(\phi_{(A, f)}) \\
&= f^!.
\end{align*}
Hence we have the morphism $\varphi^!: (B, g)^! \to (A, f)^!$ in ${\tt Quad\text{-}QHS}$. 

It is obvious that the definition of $( \ )^{!}$ on ${\tt Quad\text{-}QHS}$ respects identity maps, composition of morphisms, and reverses arrows, so \[( \ )^{!}: {\tt Quad\text{-}QHS} \to {\tt Quad\text{-}QHS}\] is a contravariant functor. This completes the construction.

\end{construction}

Recall the notion of a \emph{category with duality}, see \cite[Definition 1]{Balmer}. This is a triple $({\tt C}, \ast, \eta)$ consisting of a category $\tt C$, a contravariant functor $\ast: {\tt C} \to {\tt C}$ and a natural isomorphism $\eta: {\rm Id}_{\tt C} \to \ast \circ \ast$, such that for all objects $M \in {\tt C}$, 
\[(\eta_M)^{\ast} \circ \eta_{M^{\ast}} = {\rm id}_{M^{\ast}}.\]

For example, consider the triple $({\tt FinVect}, \ast, \eta)$, where ${\tt FinVect}$ is the full subcategory of ${\tt Vect}$ consisting of finite-dimensional vector spaces, the functor $*: {\tt FinVect} \to {\tt FinVect}$ is given by $*(V) = V^*$, the usual $\k$-linear dual of $V$, and $\eta_V: V \to (V^*)^*$ is the natural map given by sending $v \in V$ to the linear functional on $V^*$ given by evaluation at $v$. Then it is easy to check that $({\tt FinVect}, \ast, \eta)$ is a category with duality.

\begin{theorem}\label{NCQuadQHypSurf is category with duality}
Let $\eta: {\rm Id}_{{\tt Quad\text{-}QHS}} \to ( \ )^! \circ ( \ )^!$ be the natural isomorphism whose component at $(A, f) \in {\tt Quad\text{-}QHS}$ is induced by the natural isomorphism $A \to (A^!)^!$. Then the triple $({\tt Quad\text{-}QHS}, ( \ )^!, \eta)$ is a category with duality. 
\end{theorem}

\begin{proof}

Let $A = T(V)/\la I_A \ra$ be a quadratic algebra. Then \[(A^!)^! = T((V^*)^*)/\la (I_A^{\perp})^{\perp}\ra,\] where $(I_A^{\perp})^{\perp}$ is considered as a subspace of $(V^*)^* \tsr (V^*)^*$. It is well known that the algebra map $\eta_A: A \to (A^!)^!$ induced by the canonical map $V \to (V^*)^*$ is an isomorphism, and, moreover, defines a natural isomorphism $\eta: {\rm Id}_{\tt Quad} \to ( \ )^! \circ ( \ )^!$. Now consider some pair $(A, f) \in {\tt Quad\text{-}QHS}$. We will now prove that the isomorphism $\eta_A: A \to (A^!)^!$ in ${\tt Quad}$ yields an isomorphism $\eta_{(A, f)}: (A, f) \to ((A, f)^!)^! $ in ${\tt Quad\text{-}QHS}$. Recall that $(A, f)^! = (\bar{A}^!, f^!)$, and choose some $F^! \in V^* \tsr V^*$ such that $f^! = F^!+(I_{\bar{A}})^{\perp}$. Then it is easy to see that we have $(I_{\bar{A}})^{\perp} + \k F^! = I_A^{\perp}$ as subspaces of $V^* \tsr V^*$, hence \[\overline{(\bar{A}^!)} = T(V^*)/\la (I_{\bar{A}})^{\perp} + \k F^! \ra = T(V^*)/\la I_A^{\perp} \ra = A^!.\] Therefore we see that $(\bar{A}^!, f^!)^! = ((A^!)^!, (f^!)^!)$. Now we will show that the morphism $\eta_A: A \to (A^!)^!$ sends $f$ to $(f^!)^!$. Consider the following diagram 
\[\xymatrix{&A_2 \ar[r]^-{\eta_A} \ar[rd]_-{\l} & (A^!)^!_2 \ar[d]^{\psi} \\ &  & (I_A^{\perp})^*,
}
\]
where the maps are given by \[\psi: (A^!)^!_2 \to (I_A^{\perp})^*, \quad \psi(t+(I_A^{\perp})^{\perp})(r) = t(r) \text{ for all } t \in (V^*)^* \tsr (V^*)^* \text{ and } r \in I_A^{\perp};\] and \[\l: A_2 \to (I_A^{\perp})^*, \quad \l(X+I_A)(r) = r(X) \text{ for all } X \in V \tsr V \text{ and } r \in I_A^{\perp}.\] 
One can easily check that these maps are well-defined, that $\psi$ is the vector space isomorphism defined above, and that the diagram commutes. Furthermore, we can see that $\l(f) = \phi_{(\bar{A}^!, f^!)}$, where $\phi_{(\bar{A}^!, f^!)} = \psi((f^!)^!)$ is the functional defined above, so it follows that $\eta_A(f) = (f^!)^!$. Thus, in ${\tt Quad\text{-}QHS}$ we have the isomorphism $\eta_{(A, f)}: (A, f) \to ((A, f)^!)^!$, induced by $A \to (A^!)^!$. It is clear that $\eta_{(A, f)}$ is natural in $(A, f)$, whence we have a natural isomorphism of functors $\eta: {\rm Id}_{{\tt Quad\text{-}QHS}} \to ( \ )^! \circ ( \ )^!$. 

Finally, the required identity \[(\eta_{(A, f)})^! \circ \eta_{(A, f)^!} = {\rm id}_{(A, f)^!}\] 
boils down to the fact that the analogous identity \[(\eta_V)^* \circ \eta_{V^*} = {\rm id}_{V^*}\] holds in the triple $({\tt FinVect}, \ast, \eta)$, as discussed above.
\end{proof}


We now begin to narrow our focus to a certain full subcategory of ${\tt Quad\text{-}QHS}$ consisting of pairs of Koszul algebras and quadratic normal regular elements. This category is defined below, Definition \ref{category NCKosHypSurf}, and will be a main object of consideration in the following sections when we study matrix factorizations.

Given some $(A, f) \in {\tt Quad\text{-}QHS}$, we first consider the relationship between the conditions of normality and regularity for the elements $f \in A$ and $f^! \in \bar{A}^!$. We have the following interesting result.

\begin{proposition}[{\cite[Lemma 1.3]{Shelton-Tingey}}] The element $f \in A$ is normal if and only if the element $f^! \in \bar{A}^!$ is 1-regular.
\end{proposition}

We make this more precise in the following.


%

\begin{proposition}\label{normal, 1-regular}
Let $A = T(V)/\la I_A \ra$ be a quadratic algebra. Suppose that $f \in A_2$ is normal and $1$-regular, and let $\varphi: V \to V$ be the unique invertible linear map such that $fv = \varphi(v)f$ for all $v \in V$. Then the element $f^! \in \bar{A}^!_2$ is normal and $1$-regular. Moreover, $f^! \l = (\varphi^{-1})^*(\l)f^!$ for all $\l \in V^*$, where $(\varphi^{-1})^*$ is the $\k$-linear dual of $\varphi^{-1}$.
\end{proposition}

\begin{proof}
The statement that $f^!$ is normal and $1$-regular is proved in \cite[Lemma 1.3]{Shelton-Tingey}. However, note that for the implication \[f \text{ 1-regular } \implies f^! \text{ normal},\] we can use our result above that $\eta_{(A, f)}: (A, f) \to ((A^!)^!, (f^!)^!)$ is an isomorphism in ${\tt Quad\text{-}QHS}$. 

We proceed to prove the moreover statement. Fix a choice of $F \in V \tsr V$ such that $f = F+I_A$. Let $I_{\bar{A}} = I_A + \k F$ and $\bar{A} = T(V)/\la I_{\bar{A}} \ra$. Fix a choice of $F^! \in V^* \tsr V^*$ such that $f^! = F^!+(I_{\bar{A}})^{\perp}$. Let $\phi = \phi_{(A, f)} \in (I_{\bar{A}})^*$ be as above. 

We first claim that the map \[\psi:= \phi \tsr 1-\varphi^{-1} \tsr \phi: I_{\bar{A}} \tsr V \cap V \tsr I_{\bar{A}} \to V\] 
is the zero map. To see this, let $r \in I_{\bar{A}} \tsr V \cap V \tsr I_{\bar{A}}$ and write \[r = r_1+F \tsr v_1 = r_2 + v_2 \tsr F\] for some $r_1 \in I_A \tsr V$, $r_2 \in V \tsr I_A$; $v_1, v_2 \in V$. Then $\psi(r) = v_1-\varphi^{-1}(v_2)$. 

By assumption, \[v_2 \tsr F \equiv F \tsr \varphi^{-1}(v_2) \pmod{V \tsr I_A + I_A \tsr V}.\] Since $F \tsr v_1 - v_2 \tsr F = r_2 - r_1 \in V \tsr I_A + I_A \tsr V$, we have \[F\tsr(v_1-\varphi^{-1}(v_2)) \equiv F \tsr v_1 - v_2 \tsr F \equiv 0 \pmod{V \tsr I_A + I_A \tsr V}.\] So $f(v_1-\varphi^{-1}(v_2)) = 0$ in $A$, and so $v_1-\varphi^{-1}(v_2) = 0$, whence the claim. 

Now we claim that \[F^! \tsr \l-(\varphi^{-1})^*(\l) \tsr F^! \in (I_{\bar{A}})^{\perp} \tsr V^* + V^* \tsr (I_{\bar{A}})^{\perp} \quad \text{ for all } \l \in V^*.\] 
To see this, since \[[(I_{\bar{A}})^{\perp} \tsr V^* + V^* \tsr (I_{\bar{A}})^{\perp}]^{\perp} = I_{\bar{A}} \tsr V \cap V \tsr I_{\bar{A}},\] 
it is equivalent to show that \[r(F^! \tsr \l - (\varphi^{-1})^*(\l) \tsr F^!) = 0 \quad \text{ for all } r \in I_{\bar{A}} \tsr V \cap V \tsr I_{\bar{A}}.\] 
Write $r = r_1+F\tsr v_1 = r_2+v_2 \tsr F$ as above. Then 
\begin{align*}
r(F^! \tsr \l - (\varphi^{-1})^*(\l) \tsr F^!) &= \l(v_1)-[(\varphi^{-1})^*(\l)](v_2) \\
&= \l(v_1-\varphi^{-1}(v_2)) \\
&= \l(0) \quad \text{ by the first claim} \\
&= 0.
\end{align*}
This proves the second claim, and we see that $f^!\l = (\varphi^{-1})^*(\l)f^!$ in $\bar{A}^!$ for all $\l \in V^*$, as desired.
\end{proof}

In the case where $A$ is a Koszul algebra, we have the following stronger result. 
\begin{proposition}\label{normal, regular, Koszul case}
Suppose $A = T(V)/\la I_A \ra$ is Koszul and $f \in A_2$ is normal and regular with normalizing automorphism $\varphi_f: A \to A$. Then $f^! \in (\bar{A})^!_2$ is normal and regular. Furthermore, the normalizing automorphism $\varphi_{f^!}: \bar{A}^! \to \bar{A}^!$ of $f^!$ is induced by the linear map $(\varphi^{-1})^*: V^* \to V^*$.
\end{proposition}

\begin{proof}
The statement that $f^! \in (\bar{A})^!_2$ is normal and regular is proved in \cite[Corollary 1.4]{Shelton-Tingey}. The statement that the normalizing automorphism of $f^! \in (\bar{A})^!_2$ is induced by $(\varphi^{-1})^*: V^* \to V^*$ follows immediately from Proposition \ref{normal, 1-regular}.
\end{proof}

Motivated, in part, by the last result, we now define the promised full subcategory of ${\tt Quad\text{-}QHS}$. 

\begin{definition}\label{category NCKosHypSurf}
Let $\tt{Kos\text{-}QHS}$ denote the full subcategory of the category ${\tt Quad\text{-}QHS}$ whose objects are pairs $(A, f)$, where $A$ is a Koszul algebra and $f \in A_2$ is a normal, regular element. We call $\tt{Kos\text{-}QHS}$ the \emph{category of noncommutative Koszul quadric hypersurfaces}.
\end{definition}

Given some $(A, f) \in \tt{Kos\text{-}QHS}$, the well-known result that $A/\la f \ra$ is Koszul, (\cite[Theorem 1.2]{Shelton-Tingey}), implies that the algebra $\bar{A}^!$ is also Koszul. Hence, using Proposition \ref{normal, regular, Koszul case}, the contravariant functor \[( \ )^!: {\tt Quad\text{-}QHS} \to {\tt Quad\text{-}QHS}\] 
restricts to a functor $( \ )^!: \tt{Kos\text{-}QHS} \to \tt{Kos\text{-}QHS}$. The next result then follows immediately 
from Theorem \ref{NCQuadQHypSurf is category with duality}.

\begin{corollary}\label{NCKosHypSurf is category with duality}
The triple $(\tt{Kos\text{-}QHS}, ( \ )^!, \eta)$ is a category with duality.
\end{corollary}

\section{Homotopy category of curved modules and noncommutative matrix factorizations}

Suppose that $(A, f)$ is a noncommutative quadratic quadric hypersurface. In this section we first construct a functor from the category of graded modules over the algebra $\bar{A}^!$ to the homotopy category of CDG-modules over a certain CDG-algebra. More precisely, let $T = A \tsr \bar{A}^!$ and $\theta_T = f \tsr f^! \in T$. We show that there is a canonical CDG-algebra $T^{\bullet} = (T, d_T, \theta_T)$. Then we construct and prove that there is a faithful functor ${\mc F} : {\tt GrMod}\text{-}\bar{A}^! \to{\tt Ho}( {\tt CDGMod}\text{-}T^{\bullet})$. We then impose some mild restrictions on the pair $(A,f)$ and consider a certain full subcategory, denoted ${\tt P}(\bar{A}^!, f^!)$, of ${\tt GrMod}\text{-}\bar{A}^!$. We then prove that the restriction of the functor ${\mc F} : {\tt GrMod}\text{-}\bar{A}^! \to{\tt Ho}( {\tt CDGMod}\text{-}T^{\bullet})$ to ${\tt P}(\bar{A}^!, f^!)$ is valued in the category of graded left noncommutative matrix factorizations of $f$, in the sense of \cite{Mori-Ueyama-1}.


\subsection{The curved DG algebra associated to $(A, f) \in {\tt Quad}\text{-}{\tt QHS}$}

Suppose that $(A, f)$ is a noncommutative quadratic quadric hypersurface. Consider the tensor product algebra $T = A \tsr \bar{A}^!$ (so $A$ and $\bar{A}^!$ are commuting subalgebras), and endow $T$ with an $\N$-graded algebra structure via $T_n = A \tsr \bar{A}^!_n$. Suppose that $A = T(V)/\la I_A \ra$, and let $\{x_k\} \subset V$, $\{\l_k\} \subset V^*$ be a pair of dual bases. Let $b = \sum_k x_k \tsr \l_k \in T_1$. It is well known that the element $b \in T_1$ does not depend on the choice of dual bases. To see this, one can consider the canonical isomorphism $\Omega: V \tsr V^* \to \Hom(V, V)$ given by $v \tsr \l \mapsto [w \mapsto \l(w)v]$ and then check that $\Omega(b) = \id_V$. 

\begin{lemma}\label{magic square lemma}
Let $T$ and $b \in T_1$ be defined as above. Then, in the algebra $T$, we have $b^2 = f \tsr f^!$. 

\end{lemma}

\begin{proof}

Let $A = T(V) \la I_A \ra$, and recall the definitions of $\bar{A}^!$ and the isomorphism in \ref{psi_(A,f)}, 
$\psi_{(A, f)}: \bar{A}^!_2 \to I_{\bar{A}}^*$. Let $\mu: A \tsr A$ and $\mu_{\bar{A}^!}$ denote the multiplication maps in $A$ and $\bar{A}^!$, respectively. Consider the following diagram in ${\tt Vect}$

\begin{equation}
\label{magic square}
\xymatrix{&V \tsr V \tsr V^* \tsr V^* \ar[rr]^-{\mu_A \tsr \mu_{\bar{A}^!}} \ar[d]_-{\Omega} & & A_2 \tsr \bar{A}^!_2 \ar[d]^{\Psi} \\ &{\rm Hom}_{\k}(V \tsr V, V \tsr V) \ar[rr]^-{(\mu_A)_*  \p} & & {\rm Hom}_{\k}(I_{\bar{A}}, A_2)
}
\end{equation}
where the vertical maps are given by
\begin{align*}
&\Omega: v_1 \tsr v_2 \tsr \l_1 \tsr \l_2 \mapsto [u_1 \tsr u_2 \mapsto \l_1(u_1)\l_2(u_2) \, v_1 \tsr v_2], \\
&\Psi: x \tsr \t \mapsto [r \mapsto [\psi_{(A, f)}(\tau)](r) \, x],
\end{align*}
and the map $(\mu_A)_{*} \p$ is restriction to $I_{\bar{A}}$ followed by the pushforward of the multiplication map $\mu_A: V \tsr V \to A_2$. We note that the vertical maps in diagram \ref{magic square} are isomorphisms. This is well known for $\Omega$, and easy to see for $\Psi$ since $\Psi$ is the composite mapping
\[
\xymatrix{
& A_2 \tsr \bar{A}_2^! \ar[rr]^-{\id \tsr \psi_{(A,f)}} & & A_2  \tsr I_{\bar{A}}^* \ar[r] & \Hom_{\k}(I_{\bar{A}}, A_2),
}
\]
where the last map is the canonical isomorphism. 

We now check that diagram \ref{magic square} commutes. Let $v_1 \tsr v_2 \tsr \l_1 \tsr \l_2 \in V \tsr V \tsr V^* \tsr V^*$ and $r \in I_{\bar{A}}$. Then, along the east-south path we have

\begin{align*}
\big[ [\Psi \circ (\mu_A \tsr \mu_{\bar{A}^!})](v_1 \tsr v_2 \tsr \l_1 \tsr \l_2)\big](r) &= \big[\Psi\big((v_1 \tsr v_2 + I_A) \tsr (\l_1 \tsr \l_2 +I_{\bar{A}^!})\big)\big](r) \\
&= [\psi_{(A, f)}(\l_1 \tsr \l_2 + I_{\bar{A}^!})](r)\big(v_1 \tsr v_2 + I_A \big) \\
&= (\l_1 \tsr \l_2)(r) \big(v_1 \tsr v_2 + I_A\big), 
\end{align*}
and along the south-east path we have

\begin{align*}
\big[[(\mu_A)_* \p \circ \Omega](v_1 \tsr v_2 \tsr \l_1 \tsr \l_2)\big](r) &= \mu_A([\Omega(v_1 \tsr v_2 \tsr \l_1 \tsr \l_2)](r)) \\
&= (\l_1 \tsr \l_2)(r) \big(v_1 \tsr v_2 + I_A\big),
\end{align*}
showing that diagram \ref{magic square} is commutative. 

Fix dual bases $\{x_k\} \subset V$, $\{\l_k\} \subset V^*$ and consider the element \[S = \sum_{k, l} x_k \tsr x_l \tsr \l_k \tsr \l_l \in V \tsr V \tsr V^* \tsr V^*.\] It is clear that $\Omega(S) = \id_{V \tsr V}$. Recall that $I_{\bar{A}} = I_A \oplus \k F$, where $\mu_A(F) = f$. Let $r \in I_{\bar{A}}$ and write $r = t+aF$ for some $t \in I_A$ and $a \in \k$. Then unraveling definitions yields \[\big[(\mu_A)_* \p \, \Omega(S)\big](r) = af.\] 

Since diagram \ref{magic square} is commutative, we have 
\[\big[\Psi (\mu_A \tsr \mu_{\bar{A}^!})(S)\big](r) = \big[\Psi(b^2)\big](r) = af.\]

We also have
\begin{align*}
\big[\Psi(f \tsr f^!)\big](r) &= \psi_{(A, f)}(f^!)(r) \, f, \quad \text{definition of $\Psi$} \\
&= \phi_{(A, f)}(r) \, f,  \quad \ \ \ \ \   \text{definition of $f^!$ and \ref{phiAf}} \\
&= \phi_{(A, f)}(t+aF) \, f \\
&= af, \quad \ \ \ \ \ \ \ \ \ \ \ \ \ \ \ \,   \text{definition of $\phi_{(A, f)}$}.
\end{align*}

So injectivity of $\Psi$ yields the result.
\end{proof}

The next result follows immediately from Lemma \ref{magic square lemma} and Proposition \ref{easy CDGs}.

\begin{proposition}\label{T^bullet CDG}
Let $T$ and $b \in T_1$ be as defined above. Define the map $d_T: T \to T$ by \[d_T(t) = bt - (-1)^{|t|} tb \quad \text{for all homogeneous } t \in T,\] and let $\theta_T = f \tsr f^{!}$. Then $T^{\bullet} = (T, d_T, \theta_T)$ is a CDG-algebra. 
\end{proposition} 

\subsection{The functor ${\mc F} : {\tt GrMod}\text{-}\bar{A}^! \to {\tt Ho}({\tt CDGMod}\text{-}T^{\bullet})$}

We continue to consider the CDG-algebra $T^{\bullet} = (T, d_T, \theta_T)$ from Proposition \ref{T^bullet CDG}, where $T = A \tsr \bar{A}^!$, and the element $b \in T_1$, the map $d_T$ and $\theta_T = f \tsr f^!$ are as described above. We proceed to construct a functor ${\mc F} : {\tt GrMod}\text{-}\bar{A}^! \to {\tt Ho}({\tt CDGMod}\text{-}T^{\bullet})$. 

To define ${\mc F}$ on objects, consider a graded module  $M = \bigoplus_{i \in \Z} M_i$ in ${\tt GrMod}\text{-} \bar{A}^!$. Define the $\Z$-graded vector space ${\mc T}(M)$ by \[{\mc T}(M) = \bigoplus_{i \in \Z} {\mc T}(M)_i = \bigoplus_{i \in \Z} A(i) \tsr M_i.\] The reason for the somewhat strange shifted module $A(i)$ will become apparent in Section \ref{nmfs}. There is an obvious right action of $T$ on ${\mc T}(M)$ and it is clear that this makes ${\mc T}(M)$ into a graded right $T$-module. We define a degree $1$ linear map $d_{{\mc T}(M)} : {\mc T}(M) \to {\mc T}(M)$ by \[d_{{\mc T}(M)}(a \tsr m) = (-1)^{|m|+1} (a \tsr m)b \quad \text{for all } a \in A, \text{ and homogeneous } m \in M.\]
Then by Proposition \ref{easy CDG modules}, the pair $({\mc T}(M), d_{{\mc T}(M)})$ is an object of the homotopy category ${\tt Ho}({\tt CDGMod}\text{-}T^{\bullet})$ and we set ${\mc F}(M) = ({\mc T}(M), d_{{\mc T}(M)})$. 

We wish to explain, from another point of view, that the definition of $d_{{\mc T}(M)}$ is a natural one. Let $\p_i^M: M_i \tsr V^* \to M_{i+1}$ be the right action map, and recall the definition of the adjoint map $\Delta(p_i^M): M_i \to V \tsr M_{i+1}$, see subsection \ref{vector spaces}.  We define the map $\phi_i^M: \mc{T}(M)_i \to \mc{T}(M)_{i+1}$ to be the composite
\[\mc{T}(M)_i = A(i) \tsr M_i \xrightarrow{{\rm id} \tsr \Delta(\rho_i^M)} A(i) \tsr V \tsr M_{i+1} \xrightarrow{\mu_A \tsr {\rm id}} A(i+1) \tsr M_{i+1} = \mc{T}(M)_{i+1},\] 
where $\mu_A: A \tsr V \to A$ is multiplication in the algebra $A$. Using equation \ref{Delta} it is easy to check that $\phi_i^M$ is right multiplication by $b \in T_1$.  

Now we turn to defining ${\mc F}$ on morphisms. Let $M$, $M'$ be in ${\tt GrMod}\text{-}\bar{A}^!$, and suppose that $\psi: M \to M'$ is a morphism in ${\tt GrMod}\text{-}\bar{A}^!$. We consider the graded map $\id \tsr \psi: {\mc T}(M) \to {\mc T}(M')$ whose $i$th component is $\id_{A(i)} \tsr \psi_i: {\mc T}(M)_i \to {\mc T}(M')_i$. It is straightforward to check that $\id \tsr \psi \in Z^0(\Hom_T({\mc F}(M), {\mc F}(M')))$, and we define ${\mc F}(\psi) = [\id \tsr \psi]$ to be the class of $\id \tsr \psi$ in $H^0(\Hom_T({\mc F}(M), {\mc F}(M')))$.

\begin{theorem}\label{the functor F}
Let $(A, f) \in {\tt Quad}\text{-}{\tt QHS}$ and let $T^{\bullet} = (T, d_T, f \tsr f^!)$ be the CDG-algebra given in Proposition \ref{T^bullet CDG}. The construction above defines a faithful functor \[{\mc F}: {\tt GrMod}\text{-}\bar{A}^! \to {\tt Ho}({\tt CDGMod}\text{-}T^{\bullet}).\]
\end{theorem}

\begin{proof}
It is clear that ${\mc F}$ preserves identity morphisms and composition, so ${\mc F}$ is a functor. To check that ${\mc F}$ is faithful, note that both ${\tt GrMod}\text{-}\bar{A}^!$ and ${\tt Ho}({\tt CDGMod}\text{-}T^{\bullet})$ are $\k$-linear categories. Let $\psi: M \to M'$ be a morphism in ${\tt GrMod}\text{-}\bar{A}^!$ and suppose that ${\mc F}(\psi) = 0$ in ${\tt Ho}({\tt CDGMod}\text{-}T^{\bullet})$. This means that there is some map $\s \in \Hom_T({\mc F}(M), {\mc F}(M'))_{-1}$ such that $d \s = \id \tsr \psi$. So we have \[d_{{\mc F}(M')} \s + \s d_{{\mc F}(M)} = \id \tsr \psi.\] Evaluating the left hand side of the last equation at $a \tsr m \in A(i) \tsr M_i$ yields \[(-1)^i \s(a \tsr m)b + \s((-1)^{i+1}(a \tsr m)b) = 0,\] where the last equality holds since $\s$ preserves the right action of $T$ on ${\mc T}(M)$. Therefore $\id \tsr \psi = 0$ and $\psi = 0$ follows since the functor $A \tsr_{\k} \underline{\ \ } $ is exact. Hence ${\mc F}$ is faithful. 
\end{proof}

The categories ${\tt GrMod}\text{-}\bar{A}^!$ and ${\tt Ho}({\tt CDGMod}\text{-}T^{\bullet})$ are both additive $\k$-linear categories, so the following result follows immediately from the fact that the functor  ${\mc F}: {\tt GrMod}\text{-}\bar{A}^! \to {\tt Ho}({\tt CDGMod}\text{-}T^{\bullet})$ is faithful.

\begin{corollary}\label{F is nontrivial}
Let $(A, f) \in {\tt Quad}\text{-}{\tt QHS}$ and let $T^{\bullet} = (T, d_T, f \tsr f^!)$ be the CDG-algebra given above. For $M \in {\tt GrMod}\text{-}\bar{A}^!$, ${\mc F}(M) = 0$ in ${\tt Ho}({\tt CDGMod}\text{-}T^{\bullet})$ if and only if $M = 0$ in ${\tt GrMod}\text{-}\bar{A}^!$.
\end{corollary}

Let ${\tt CDGMod}\text{-}T^{\bullet}_b$ be the full subcategory of ${\tt CDGMod}\text{-}T^{\bullet}$ whose objects are the right CDG-modules of the form $(N, d_{N, b})$ (as described in Proposition \ref{easy CDG modules}), where $b \in T_1$ is defined as above. Note that the functor ${\mc F}$ from Theorem \ref{the functor F} has its values in the category ${\tt CDGMod}\text{-}T^{\bullet}_b$ so we may consider ${\mc F}: {\tt GrMod}\text{-}\bar{A}^! \to {\tt CDGMod}\text{-}T^{\bullet}_b$. As we will now prove, this functor is left adjoint to a functor ${\mc G}: {\tt CDGMod}\text{-}T^{\bullet}_b \to {\tt GrMod}\text{-}\bar{A}^!$. 

We now proceed to define ${\mc G}$. For $(N, d_{N, b}) \in {\tt CDGMod}\text{-}T^{\bullet}_b$ define \[{\mc G}(N, d_{N, b}) = \Hom_T(T, N) = \bigoplus_{i \in \Z} \Hom_{{\tt GrMod}\text{-}T}(T, N(i)).\]
Then ${\mc G}(N, d_{N, b})$ is endowed with the structure of a graded right $\bar{A}^!$-module as follows. For $\phi \in \Hom_{{\tt GrMod}\text{-}T}(T, N(i))$ and $a^! \in \bar{A}^!$, using the canonical left action of $\bar{A}^!$ on $T$, define \[\phi a^!: T \to N \quad (\phi a^!)(t) = \phi(a^! t), \quad \text{for all } t \in T.\] 
If $\psi: (N_1, d_{N_1, b}) \to (N_2, d_{N_2, b})$ is a morphism in ${\tt CDGMod}\text{-}T^{\bullet}_b$, then we define \[{\mc G}(\psi): {\mc G}(N_1, d_{N_1, b}) \to {\mc G}(N_2, d_{N_2, b}), \quad {\mc G}(\psi)(\phi) = \psi \circ \phi \quad \text{for all } \phi \in {\mc G}(N_1, d_{N_1, b}).\] One now easily checks that this defines a functor ${\mc G}: {\tt CDGMod}\text{-}T^{\bullet}_b \to {\tt GrMod}\text{-}\bar{A}^!$.

To check that $({\mc F}, {\mc G})$ is an adjoint pair of functors, define maps $\tau$ and $\sigma$ as in 
\[\xymatrix{ \Hom_{{\tt CDGMod}\text{-}T^{\bullet}_b}({\mc F}(M), (N, d_{N, b})) \ar@<.5ex>[r]^-{\tau}  & \Hom_{{\tt GrMod}\text{-}\bar{A}^!}(M, {\mc G}(N, d_{N, b})) \ar@<.5ex>[l]^-{\sigma}}\]
by \begin{align*} 
[\tau(\phi)(m)](a \tsr a^!) &= \phi(a \tsr ma^!) \quad \text{for all } m \in M, a \in A, a^! \in \bar{A}^! \\
[\s(\psi)](a \tsr m) &= \psi(m)(a \tsr 1) \quad \text{for all } a \in A, m \in M.
\end{align*}
It is then straightforward to check that $\t$ and $\s$ are well-defined mutually inverse natural isomorphisms. 

\begin{proposition}\label{adjoint pair}
As defined above, the functors ${\mc F}: {\tt GrMod}\text{-}\bar{A}^! \to {\tt CDGMod}\text{-}T^{\bullet}_b$ and ${\mc G}: {\tt CDGMod}\text{-}T^{\bullet}_b \to {\tt GrMod}\text{-}\bar{A}^!$  form an adjoint pair $({\mc F}, {\mc G})$ of functors.  
\end{proposition}

\begin{question}\label{muse on derived equivalence}
We have checked that the functor ${\mc G}: {\tt CDGMod}\text{-}T^{\bullet}_b \to {\tt GrMod}\text{-}\bar{A}^!$ also induces a functor ${\mc G}: {\tt Ho}({\tt CDGMod}\text{-}T^{\bullet}_b) \to {\tt GrMod}\text{-}\bar{A}^!$. It would interesting if the pair $({\mc F}, {\mc G})$ can be upgraded to an equivalence of categories by passing to some kind of derived categories. Based on noncommutative algebraic geometry, one could try upgrading ${\tt GrMod}\text{-}\bar{A}^!$ to the bounded derived category ${\rm D}^{\rm b}({\tt Tails} \, \bar{A}^!)$. It is less clear how one might go about ``deriving" the category ${\tt CDGMod}\text{-}T^{\bullet}$ since the differential, $d_N$, in a CDG-module, $(N, d_N)$, need not square to zero, hence the notion of a quasi-isomorphism of CDG-modules does not make sense. Positselski and Becker have proposed several candidates for derived categories of CDG-modules (or categories of matrix factorizations), see \cite{Becker, Positselski}. So we have the following question. Can the functor ${\mc F}: {\tt GrMod}\text{-}\bar{A}^! \to {\tt Ho}({\tt CDGMod}\text{-}T^{\bullet})$ be extended to a category equivalence \[{\mc F}: {\rm D}^{\rm b}({\tt Tails} \, \bar{A}^!) \to {\rm D}({\tt CDGMod}\text{-}T^{\bullet}),\] where ${\rm D}({\tt CDGMod}\text{-}T^{\bullet})$ is one of Positselski's or Becker's candidates?
\end{question}

\subsection{Noncommutative graded left matrix factorizations}\label{nmfs}

In this section we consider $(A, f) \in {\tt Quad}\text{-}{\tt QHS}$ and assume that the algebra $A$ satisfies the strong left rank condition, and that the map $\p_f: A \to A$ is injective. 
\begin{definition}\label{category I}
Let ${\tt P}(\bar{A}^!, f^!)$ denote the full subcategory of ${\tt GrMod}\text{-}\bar{A}^!$ consisting of graded right locally finite-dimensional $\bar{A}^!$-modules $M = \oplus_{i \in \Z} M_i$ such that the map $\p_{f^!}: M_i \to M_{i+2}$ is a $\k$-linear isomorphism for all $i \in \Z$.
\end{definition}
We will now explain how the functor ${\mc F}: {\tt GrMod}\text{-}\bar{A}^! \to {\tt Ho}({\tt CDGMod}\text{-}T^{\bullet})$ from Theorem \ref{the functor F} may also be considered to be a functor 
${\mc F}: {\tt P}(\bar{A}^!, f^!) \to {\tt NMF}_{\Z}^A(f)$, where ${\tt NMF}_{\Z}^A(f)$ is the category of noncommutative graded left matrix factorizations of $f$ as defined in \cite{Mori-Ueyama-1}. We now proceed to give the definition of the category ${\tt NMF}_{\Z}^A(f)$.

Let ${\tt Fun}(\Z, A\text{-}{\tt GrMod})$ denote the usual functor category, where $\Z$ is the poset ($\leq$) category of the integers. So an object of ${\tt Fun}(\Z, A\text{-}{\tt GrMod})$ is a sequence of graded left $A$-module homomorphisms, $\{\phi_i : F_i \to F_{i+1}\}_{i \in \Z}$, and a morphism of ${\tt Fun}(\Z, A\text{-}{\tt GrMod})$, say $\eta: \{\phi_i : F_i \to F_{i+1}\}_{i \in \Z} \to \{\g_i : G_i \to G_{i+1}\}_{i \in \Z}$, is a sequence of graded left $A$-module homomorphisms $\eta = \{\eta_i: F_i \to G_i\}_{i \in \Z}$ such that the diagram
\[\xymatrix{&F_i \ar[r]^-{\phi_i} \ar[d]_-{\eta_i} & F_{i+1} \ar[d]^{\eta_{i+1}} \\ & G_i \ar[r]^-{\g_{i}} & G_{i+1}
}
\]
commutes for all $i \in \Z$. 

\begin{definition}\label{category of matrix factorizations}
Let $(A, f) \in {\tt Quad\text{-}QHS}$. Following \cite{Mori-Ueyama-1}, let ${\tt NMF}_{\Z}^A(f)$ denote the full subcategory of ${\tt Fun}(\Z, A\text{-}{\tt GrMod})$ whose objects are sequences of graded left $A$-module homomorphisms $\{\phi_i : F_i \to F_{i+1}\}_{i \in \Z}$ such that there exists some $r \in \N$ and graded left $A$-module isomorphisms $\{\nu_i: F_i \to \oplus_{j = 1}^rA(-d_{i, j})\}_{i \in \Z}$, for some $d_{i, j} \in \Z$, so that the diagram 
\[\xymatrix{&F_i \ar[r]^-{\phi_{i+1} \phi_i} \ar[d]_-{\nu_i} & F_{i+2} \ar[d]^{\nu_{i+2}} \\ & \oplus_{j= 1}^r A(-d_{i,j}) \ar[r]^-{\p_f} & \oplus_{j=1}^r A(-d_{i, j} + 2)
}
\]
commutes for all $i \in \Z$. An object of ${\tt NMF}_{\Z}^A(f)$ is called a \emph{noncommutative graded left matrix factorization of $f$}. 
\end{definition}


%

As we will now prove, the functor ${\mc F}: {\tt GrMod}\text{-}\bar{A}^! \to {\tt Ho}({\tt CDGMod}\text{-}T^{\bullet})$ from Theorem \ref{the functor F} determines noncommutative graded left matrix factorizations of $f$ in the following  precise sense. 

\begin{theorem}
\label{F functor valued in NMF}
Let $(A, f ) \in {\tt Quad\text{-}QHS}$ and assume that $A$ satisfies the left strong rank condition, and that the map $\p_{f} : A \to A$ is injective. Recall the category ${\tt P}(\bar{A}^!, f^!)$ from Definition \ref{category I}. Then there is a functor ${\mc F}: {\tt P}(\bar{A}^!, f^!) \to {\tt NMF}_{\Z}^A(f)$. 
\end{theorem}

\begin{proof}

It is convenient to define a function $\s: \Z \to \Z$ by \[\s(i) = \begin{cases} i/2 & i \equiv 0(2)\\
(i-1)/2 & i \equiv 1(2).
\end{cases}
\] Then $\s(i+2) = \s(i)+1$ for all $i \in \Z$. 

Let $M = \oplus_{i \in \Z} M_i$ be an object of ${\tt P}(\bar{A}^!, f^!)$. By assumption, there exist $r, s \in \N$ such that \[\dim_{\k} M_i = \begin{cases} r & i \equiv 0(2) \\
s & i \equiv 1(2).
\end{cases}
\] 
Fix bases $\{m_{0, j}\}_{1 \leq j \leq r}$ and $\{m_{1, j}\}_{1 \leq j \leq s}$ for $M_0$ and $M_1$, respectively. Recall from above that ${\mc T}(M)_i = A(i) \tsr M_i$ and the map $\phi_i^M: {\mc T}(M)_i \to {\mc T}(M)_{i+1}$ is right multiplication by $b = \sum_k x_k \tsr \l_k \in A \tsr \bar{A}^!$. We consider ${\mc T}(M)_i$ to be a $\Z$-graded vector space by declaring that its $n$th graded piece is \[({\mc T}(M)_i)_n = (A(i))_n \tsr M_i.\] It is then clear that ${\mc T}(M)_i$ is a graded free finitely generated left $A$-module under the obvious left action of $A$. Moreover, the map $\phi_i^M$ is a degree $0$ morphism of graded left $A$-modules, i.e., $\phi_i^M$ is a morphism in $A\text{-}{\tt GrMod}$, which finally explains the choice of the shifted module $A(i)$ in the definition of ${\mc T}(M)_i$.

 We have the following left $A$-basis for ${\mc T}(M)_i$: $\{1 \tsr \p_{f^!}^{\s(i)} m_{0, j}\}_{1 \leq j \leq r}$ for $i \equiv 0(2)$ and $\{1 \tsr \p_{f^!}^{\s(i)} m_{1, j}\}_{1 \leq j \leq s}$ for $i \equiv 1(2)$. Using this basis we define a left $A$-module homomorphism \[\nu_i: {\mc T}(M)_i \to \bigoplus_{j} A(i), \quad \nu_i(1 \tsr \p_{f^!}^{\s(i)} m_{k, j}) = e_{i, j}, \quad \text{for all } j,\]
where $k = 0, 1$ and $e_{i, j} \in \oplus_j A(i)$ is the $j$th standard basis vector. It is clear that $\nu_i$ is a graded left $A$-module isomorphism for all $i \in \Z$. Furthermore, using the fact that $\phi_{i+1}^M \phi_{i}^M = \p_{f \tsr f^!}$, see Lemma \ref{magic square lemma}, it is routine to check that the diagram 
\[\xymatrix{& {\mc T}(M)_i \ar[r]^-{\phi_{i+1}^M \phi_i^M} \ar[d]_-{\nu_i} &  {\mc T}(M)_{i+2} \ar[d]^{\nu_{i+2}} \\ &  \bigoplus_{j}A(i) \ar[r]^-{\p_{f}} &  \bigoplus_{j} A(i+2)
}
\]
commutes for all $i \in \Z$. 

To conclude that the sequence $\{\phi_i^M: {\mc T}(M)_i \to {\mc T}(M)_{i+1}\}_{i \in \Z}$ is a noncommutative graded left matrix factorization of $f$ we need to prove that $r = s$. The assumption that $\p_f: A \to A$ is injective together with the commutativity of the last diagram shows that $\phi_i^M: {\mc T}(M)_i \to {\mc T}(M)_{i+1}$ is an injective morphism of finite rank free left $A$-modules. Then the assumption that $A$ has the left strong rank condition yields $r \leq s$ and $s \leq r$. We set \[{\mc F}(M) = \{\phi_i^M: {\mc T}(M)_i \to {\mc T}(M)_{i+1}\}_{i \in \Z} \in {\tt NMF}_{\Z}^A(f).\]

As in the construction above, for a morphism $\psi: M \to M'$ in ${\tt GrMod}\text{-}\bar{A}^!$, the collection of maps ${\mc F}(\psi) = \{\id \tsr \psi_i: {\mc T}(M)_i \to {\mc T}(M)_{i+1}\}_{i \in \Z}$ is a morphism in ${\tt NMF}_{\Z}^A(f)$. It is then clear that we have defined a functor \[{\mc F}: {\tt P}(\bar{A}^!, f^!) \to {\tt NMF}_{\Z}^A(f).\]


\end{proof}

\begin{remark}\label{category of fat points}
If $(A, f) \in {\tt Quad}\text{-}{\tt QHS}$, the algebra $A$ satisfies the left strong rank condition, and the map $\p_{f}: A \to A$ is injective, then the proof of Theorem \ref{F functor valued in NMF} shows that $\dim_{\k} M_i$ is a constant function of $i$ for objects $M = \bigoplus_{i \in \Z} M_i$ in the category ${\tt P}(\bar{A}^!, f^!)$. Recall that for a connected graded algebra, graded modules for which the dimensions of the graded pieces are constant are sometimes called \emph{point modules} or \emph{fat point modules}. This is why we chose to use ${\tt P}$ in the notation ${\tt P}(\bar{A}^!, f^!)$, to draw a connection to \emph{fat point modules} over the algebra $\bar{A}^!$.
\end{remark}

In \cite{Mori-Ueyama-1} the authors also defined a certain quotient category of ${\tt NMF}_{\Z}^A(f)$ in which trivial matrix factorizations become isomorphic to the zero object. We adapt the definition of this quotient category in the following way. Let $F \in A\text{-}{\tt grmod}$ be a free module. Define objects $\phi^F = \{\phi^F_i\}_{i \in \Z}$ and $^F \phi = \{^F \phi_i\}_{i \in \Z}$ of ${\tt NMF}_{\Z}^A(f)$ by 
\[\phi_i^F = \begin{cases} \id: F(i) \to F(i) \quad &i \equiv 0(2) \\ \p_f: F(i-1) \to F(i+1) \quad & i\equiv 1(2), \end{cases}\]
\[^F\phi_i = \begin{cases} \p_f: F(i) \to F(i+2) \quad &i \equiv 0(2) \\ \id: F(i+1) \to F(i+1) \quad & i\equiv 1(2). \end{cases}\]
Recall that ${\tt NMF}_{\Z}^A(f)$ is a $\k$-linear additive category, so finite direct sums exist, and such sums are also direct products. Consider the full subcategory of ${\tt NMF}_{\Z}^A(f)$ defined by \[{\tt H} = \big\{\phi^F \oplus {^G\phi} : F, G \in A\text{-}{\tt grmod} \text{ are free}\big\}.\] By definition, the quotient category ${\underline {\tt NMF}}_{\Z}^A(f) = {\tt NMF}_{\Z}^A(f)/{\tt H}$ has the same objects as ${\tt NMF}_{\Z}^A(f)$ with morphisms given by \[\Hom_{{\underline {\tt NMF}}_{\Z}^A(f)}(\phi, \phi') = \Hom_{{\tt NMF}_{\Z}^A(f)}(\phi, \phi')/{\tt H}(\phi, \phi'),\] where ${\tt H}(\phi, \phi')$ is the subgroup of morphisms $\phi \to \phi'$ factoring through an object in ${\tt H}$. We now consider the functor ${\mc F}: {\tt P}(\bar{A}^!, f^!) \to {\tt NMF}_{\Z}^A(f)$ constructed in Theorem \ref{F functor valued in NMF}. By abuse of notation we write ${\mc F}: {\tt P}(\bar{A}^!, f^!) \to {\underline {\tt NMF}}_{\Z}^A(f)$ for the composition of ${\mc F}$ with the canonical quotient functor. 

\begin{proposition}\label{F functor valued in stable NMF}
Let $(A, f ) \in {\tt Quad\text{-}QHS}$ and assume that $A$ satisfies the left strong rank condition, and that the map $\p_{f} : A \to A$ is injective. Then the functor ${\mc F}: {\tt P}(\bar{A}^!, f^!) \to {\underline {\tt NMF}}_{\Z}^A(f)$ is faithful. 
\end{proposition} 

\begin{proof}
Let $\psi: M \to M'$ be a morphism in ${\tt P}(\bar{A}^!, f^!)$. We suppose that there exists some free module $F \in A\text{-}{\tt grmod}$ such that ${\mc F}(\psi): {\mc F}(M) \to {\mc F}(M')$ factors through the object $\phi^F \in {\tt NMF}_{\Z}^A(f)$. Then by definition we have the following commutative diagram for all $i \in \Z$, $i \equiv 0(2)$,
\[
\xymatrix{ 
 A(i-1) \tsr M_{i-1} \ar[d]_-{\pi_{i-1}} \ar[r]^-{\p_b} & A(i) \tsr M_i \ar[d]^-{\pi_{i+1} \p_b} \ar[r]^-{\p_b} & A(i+1) \tsr M_{i+1} \ar[d]^-{\pi_{i+1}}  \\
 F(i-2) \ar[d]_-{\iota_{i-1}} \ar[r]^-{\p_f} & F(i) \ar[d]^-{\iota_i} \ar[r]^-{\id} & F(i) \ar[d]^-{\p_b \iota_i}  \\
 A(i-1) \tsr M_{i-1}' \ar[r]^-{\p_b} & A(i) \tsr M_i' \ar[r]^-{\p_b} & A(i+1) \tsr M_{i+1}'
}
\]
where all maps are degree $0$ graded left $A$-module homomorphisms. The module $A(i+1) \tsr M_{i+1}$ is generated in degree $-(i+1)$ and the module $A(i) \tsr M_i'$ is zero in degree $-(i+1)$, hence the map $\iota_i \, \pi_{i+1}: A(i+1) \tsr M_{i+1} \to A(i) \tsr M_i'$ is the zero map. Therefore the map $\p_b \, \iota_i \, \pi_{i+1}: A(i+1) \tsr M_{i+1} \to A(i+1) \tsr M_{i+1}'$ is also the zero map. Note that by assumption, we have $\id \tsr \psi_{i+1} = p_b \, \iota_i \, \pi_{i+1}$, hence $\id \tsr \psi_{i+1} = 0$. Consider the commutative diagram
\[
\xymatrix{ A(i) \tsr M_i \ar[d]_-{\id \tsr \psi_i} \ar[r]^-{\p_b} & A(i+1) \tsr M_{i+1} \ar[d]^-{0} \\
 A(i) \tsr M_i' \ar[r]^-{\p_b} & A(i+1) \tsr M_{i+1}'.}
\]
The proof of Theorem \ref{F functor valued in NMF} shows that $\p_b$ is injective, hence $\id \tsr \psi_i = 0$. We conclude that $1 \tsr \psi = 0$, and then $\psi = 0$ follows from the fact that the functor $A \tsr_{\k} \underline{\ \ } $ is exact.

Similarly, one shows that if ${\mc F}(\psi)$ factors through the object $^G \phi \in {\tt NMF}_{\Z}^A(f)$ for some free module $G \in A\text{-}{\tt grmod}$, then $\psi = 0$. Finally, $\phi^F \oplus {^G \phi}$ for free modules $F, G \in A\text{-}{\tt grmod}$ is both a coproduct and a product in the additive category ${\tt NMF}_{\Z}^A(f)$, so it follows that if ${\mc F}(\psi)$ factors through an object in ${\tt H}$, then $\psi = 0$. We conclude that ${\mc F}: {\tt P}(\bar{A}^!, f^!) \to {\underline {\tt NMF}}_{\Z}^A(f)$ is faithful. 
\end{proof}

As a consequence of Proposition \ref{F functor valued in stable NMF} we have the next result. The proof is omitted since it follows immediately from the fact that ${\mc F}: {\tt P}(\bar{A}^!, f^!) \to {\underline {\tt NMF}}_{\Z}^A(f)$ is faithful, and ${\tt P}(\bar{A}^!, f^!)$, ${\underline {\tt NMF}}_{\Z}^A(f)$ are both additive $\k$-linear categories. 

\begin{corollary}\label{stable F is nontrivial}
Let $(A, f ) \in {\tt Quad\text{-}QHS}$ and assume that $A$ satisfies the left strong rank condition, and that the map $\p_{f} \colon A \to A$ is injective. Consider the functor ${\mc F}: {\tt P}(\bar{A}^!, f^!) \to {\underline {\tt NMF}}_{\Z}^A(f)$. For $M \in {\tt P}(\bar{A}^!, f^!)$, we have ${\mc F}(M) = 0$ in ${\underline {\tt NMF}}_{\Z}^A(f)$ if and only if $M = 0$ in ${\tt P}(\bar{A}^!, f^!)$.
\end{corollary} 

Theorem \ref{F functor valued in NMF}, Proposition \ref{F functor valued in stable NMF} and Corollary \ref{stable F is nontrivial} show that the regularity of the elements $f \in A$ and $f^! \in \bar{A}^!$ is related to constructing noncommutative graded left matrix factorizations of $f$ from objects of ${\tt GrMod}\text{-}\bar{A}^!$. In the following section we discuss this relationship in more detail.

\section{Noncommutative matrix factorizations from finite-dimensional representations}

In this section we consider noncommutative Koszul quadric hypersurfaces and prove three main results. Associated to any $(A, f) \in {\tt Kos}\text{-}{\tt QHS}$, there are two important algebras ${\mathcal Cl}(A, f)$ and $\mc{C}(A, f)$, referred to, respectively, as the \emph{Clifford algebra} and the \emph{even Clifford algebra} of the pair $(A,f)$. First, we prove that the construction of the Clifford algebra yields a contravariant functor ${\mathcal Cl}: \tt{Kos\text{-}QHS} \to \tt{GrAlg}$.  Second, we use the results of Section 4 to prove that for certain pairs $(A,f)$, there is a faithful functor ${\mc F}: {\tt fdim}\text{-}\mc{C}(A, f) \to {\underline{\tt NMF}}_{\Z}^A(f)$. Finally, we prove that when $A$ is a noetherian algebra of finite global dimension, then the even Clifford algebra $\mc{C}(A, f)$ is a PBW-deformation of a Zhang twist of the $2$-Veronese subalgebra of the quadratic dual algebra $A^!$. 

\subsection{The Clifford algebra and even Clifford algebra associated to $(A, f)$}

\begin{definition}\label{C(A,f) definition}
Let $(A, f) \in \tt{Kos\text{-}QHS}$. By Proposition \ref{normal, regular, Koszul case}, the element $f^! \in \bar{A}^!$ is normal and regular, so the set $\{(f^!)^n : n \in \N\}$ is a right Ore set in $\bar{A}^!$. Hence, we define the algebra \[{\mathcal Cl}(A, f) = \bar{A}^![(f^!)^{-1}]\] to be the associated right ring of fractions. 
\end{definition}

When the algebra $A$ is commutative, the work \cite{BEH}, seminal for the study of matrix factorizations on commutative quadrics, reveals important connections between ${\mathcal Cl}(A, f)$ and classical Clifford algebras of quadratic forms. For this reason, we will call ${\mc Cl}(A, f)$ the \emph{Clifford algebra} of the pair $(A, f)$. 

The algebra ${\mathcal Cl}(A, f)$ is naturally $\Z$-graded. For any $i \in \Z$, the  $i$th graded component of ${\mc Cl}(A, f)$ is given by \[{\mathcal Cl}(A, f)_i = \big\{x(f^!)^{-k} : k \in \Z, \ x \in \bar{A}^!_{i+2k}\big\}.\] 
In fact, by \cite[Lemma 4.4]{Mori-Ueyama-2} the $\Z$-graded algebra ${\mc Cl}(A, f)$ is strongly graded. 

Furthermore, if $(A, f) \to (B, g)$ is a morphism in $\tt{Kos\text{-}QHS}$, then the morphism $(\bar{B}^!, g^!) \to (\bar{A}^!, f^!)$ induces a canonical morphism ${\mathcal Cl}(B, g) \to {\mathcal Cl}(A, f)$ in $\tt{GrAlg}$. It is then straightforward to prove the following result.

\begin{proposition}\label{Clifford functor}
Using the construction above, we have defined a contravariant functor ${\mathcal Cl}: \tt{Kos\text{-}QHS} \to \tt{GrAlg}$. 
\end{proposition}

The degree $0$ subalgebra of ${\mathcal Cl}(A, f)$ will be very important in the sequel. 

\begin{definition}\label{Cl(A,f) definition}
Let $\mc{C}(A,f) = {\mc Cl}(A, f)_0$ be the degree $0$ subalgebra of ${\mathcal Cl}(A, f)$. We note that $\mc{C}(A,f)$ is often written as $C(A/\la f \ra)$ by other authors, \cite{Mori-Ueyama-1, Mori-Ueyama-2, Smith-vdB}. Due to the functorial emphasis of this paper, we use the notation $\mc{C}(A,f)$. We will refer to ${\mc C}(A, f)$ as the \emph{even Clifford algebra} of the pair $(A, f)$.
\end{definition}

The following result gives conditions ensuring that the graded pieces of ${\mc Cl}(A,f)$ are finite-dimensional over $\k$, as well as determining their dimensions. The proof closely follows that of \cite[Lemma 5.1(3)]{Smith-vdB}, but because we need a little more generality we include it. 

\begin{proposition}[{\cite[Lemma 5.1(3)]{Smith-vdB}, \cite[Lemma 4.7]{Mori-Ueyama-2}}]\label{Cl(A,f) finite-dimensional}

Let $(A, f) \in \tt{Kos\text{-}QHS}$ with $A$ a noetherian Koszul algebra of finite global dimension.  Then for all $i \in \Z$ the following statements hold:
\begin{itemize}
\item[(1)] ${\mc Cl}(A, f)_i$ is finite-dimensional over $\k$,
\item[(2)]  $\dim_{\k} {\mc C}(A, f) = \dim_{\k} A^!_{(2)}$,
\item[(3)]  $\dim_{\k} {\mc Cl}(A, f)_1 = \dim_{\k} \bigoplus_{n \geq 0} A^!_{2n+1}$, and
\item[(4)] $\p_{f^{!}}\colon {\mc Cl}(A,f)_i \to {\mc Cl}(A, f)_{i+2}$ is a vector space isomorphism.
\end{itemize}
\end{proposition}

\begin{proof}
Statement (4) is obvious since the inverse of $\p_{f^!}$ is $\p_{(f^!)^{-1}}$. Now we consider (2) and (3). Since $A$ is Koszul of finite global dimension the algebra $A^!$ is finite-dimensional over $\k$. Recall that $A^! \cong \bar{A}^!/\la f^! \ra$. Since $f^! \in \bar{A}^!_2$ is regular we have an equality in terms of Hilbert series \[H_{A^!}(t) = (1-t^2)H_{\bar{A}^!}(t).\] Since $A^!$ is finite-dimensional it follows that $\bar{A}_{m+2}^! = \bar{A}^!_m f^!$ for all  $m$ sufficiently large. Note that for all $i \geq 0$ it is clear that $\bar{A}^!_i \subseteq \bar{A}^!_{i+2} (f^!)^{-1}$ so we have \[{\mc Cl}(A, f)_1 = \bar{A}^!_1 + \bar{A}^!_3 (f^!)^{-1} + \bar{A}^!_5 (f^!)^{-2} + \cdots = \bar{A}^!_{2m+1}(f^!)^{-m}\] for all $m$ sufficiently large. 
Hence, for all $m$ sufficiently large, we have \[\dim_{\k} {\mc Cl}(A, f)_1 = \dim_{\k} \bar{A}^!_{2m+1} = \dim_{\k} \bigoplus_{n \geq 0} A^!_{2n+1},\] where the second equality follows from the Hilbert series equation. The proof of statement (2) for ${\mc Cl}(A, f)_0 = {\mc C}(A, f)$ is similar. 

Finally, statement (1) is immediate from (2), (3) and (4).
\end{proof} 

\subsection{The functor ${\mc F}:{\tt fdim}\text{-} \mc{C}(A, f) \to \underline{{\tt NMF}}_{\Z}^A(f)$} 

\begin{construction}\label{module to objects}
Assume that $A$ is a noetherian Koszul algebra of finite global dimension. Since $A$ is left noetherian, we note that $A$ satisfies the left strong rank condition, see \cite[Theorem 1.35]{Lam}. We also assume that $(A, f) \in {\tt Kos}\text{-}{\tt QHS}$, so by definition, $f \in A_2$ is normal and regular. 

Let $N \in {\tt{fdim}}\text{-}\mc{C}(A, f)$. We define a $\Z$-graded vector space by \[\widehat{N} = \bigoplus_{i \in \Z} \widehat{N}_i = \bigoplus_{i \in \Z} N \tsr_{\mc{C}(A, f)} {\mc Cl}(A, f)_i.\]
We consider $\widehat{N}$ to be an object of ${\tt GrMod}\text{-}\bar{A}^!$ via the natural inclusion $\bar{A}^! \to {\mc Cl}(A, f)$ of $\Z$-graded algebras. 

\begin{proposition}\label{hat N to factorization}
Let $\widehat{N} \in {\tt GrMod}\text{-}\bar{A}^!$ be as in the previous paragraph. Consider the functor $\mc{F}: {\tt P}(\bar{A}^!, f^!) \to {\tt NMF}_{\Z}^A(f)$ from Theorem \ref{F functor valued in NMF}. Then $\widehat{N} \in {\tt P}(\bar{A}^!, f^!)$, so $\mc{F}(\widehat{N}) \in {\tt NMF}_{\Z}^A(f)$. 
\end{proposition}

\begin{proof}
It follows immediately from Proposition \ref{Cl(A,f) finite-dimensional}(1) that $\widehat{N}$ is locally finite-dimensional. Furthermore, since $f^! \in {\mc Cl}(A, f)$ is regular, the map $\p_{f^!}: \widehat{N}_i \to \widehat{N}_{i+2}$ is a vector space isomorphism by Proposition \ref{Cl(A,f) finite-dimensional}(4). Hence, referring to Definition \ref{category I}, we conclude that $\widehat{N}$ is in the category ${\tt P}(\bar{A}^!, f^!)$. 
\end{proof}

Now suppose that $\psi: M \to N$ is a morphism in ${\tt fdim}\text{-}{\mc C}(A, f)$. We define the map $\widehat{\psi}: \widehat{M} \to \widehat{N}$ by declaring that its $i$th component is given by \[\widehat{\psi}_i = \psi \tsr \id: M \tsr_{{\mc C}(A, f)} {\mc Cl}(A,f )_i \to N \tsr_{{\mc C}(A, f)}{\mc Cl}(A,f )_i.\] We note that since the Clifford algebra ${\mc Cl}(A, f)$ is strongly graded the correspondence $N \to \widehat{N}$ yields an equivalence of categories ${\tt Mod}\text{-}{\mc C}(A, f) \to {\tt GrMod}\text{-}{\mc Cl}(A, f)$, see \cite[Theorem I.3.4]{NOy}. 

 It is clear that $\widehat{\psi}: \widehat{M} \to \widehat{N}$ is a morphism in ${\tt GrMod}\text{-} \bar{A}^!$ upon restriction. We may then apply the functor $\mc{F}: {\tt P}(\bar{A}^!, f^!) \to {\tt NMF}_{\Z}^A(f)$ to $\widehat{\psi}$ to obtain the morphism ${\mc F}(\widehat{\psi}): {\mc F}(\widehat{M}) \to {\mc F}(\widehat{N})$ in ${\tt NMF}_{\Z}^A(f)$.  It is now routine to prove the following result. 
\begin{theorem}\label{NMFs from reps}
Let $A$ be a noetherian Koszul algebra of finite global dimension, and let $(A, f) \in {\tt Kos}\text{-}{\tt QHS}$. The construction above defines a faithful functor \[{\mc F}:{\tt fdim}\text{-} \mc{C}(A, f) \to \underline{{\tt NMF}}_{\Z}^A(f).\]
\end{theorem}

\end{construction}

In order to make Theorem \ref{NMFs from reps} effective it would be useful to present the even Clifford algebra ${\mc C}(A, f)$ in terms of generators and relations based on the data $(A, f)$. We now give a solution to this problem. 


\subsection{The even Clifford algebra as a PBW-deformation}

When the algebra $A$ has finite global dimension, Proposition \ref{Cl(A,f) finite-dimensional} (2) begs the question whether there is a more precise relationship between the algebras $\mc{C}(A,f)$ and  $A^!_{(2)}$. In the case where $f \in A_2$ is a \emph{central} element, such a relationship was uncovered in \cite[Section 5]{He-Ye} in the context of the notion of \emph{Clifford deformations}. The Clifford deformations in \cite{He-Ye} are special cases of more general notions, namely \emph{nonhomogenenous quadratic algebras of PBW type} in the sense of \cite{Braverman-Gaitsgory}, and \emph{PBW-deformations} in the sense of \cite{Cassidy-Shelton}. 


Following \cite{Braverman-Gaitsgory}, let $V \in {\tt Vect}$ be finite-dimensional and let $P$ be a subspace of $\k \oplus V \oplus (V \tsr V)$. Then the \emph{nonhomogenous quadratic algebra associated to the pair $(V, P)$} is the algebra $Q(V, P) = T(V)/ \la P \ra$. Consider the canonical filtration on the tensor algebra given by $F^i(T(V)) = \oplus_{j \leq i} T(V)_j$. This induces a filtration on the algebra $Q(V, P)$ and we let ${\rm gr \,} Q(V, P)$ denote the associated graded algebra. Consider the canonical projection $\pi_2: T(V) \to V \tsr V$, and let $R = \pi_2(P)$. Then the \emph{quadratic algebra associated to the pair $(V, P)$} is the algebra $Q(V, R) = T(V)/\la R \ra$. It is clear that there is a natural epimorphism $Q(V, R) \to {\rm gr \,} Q(V, P)$; when this map is an isomomorphism, the algebra $Q(V, P)$ is said to be of \emph{PBW type}.
In the language of \cite{Cassidy-Shelton}, one refers to the algebra $Q(V, P)$ as a \emph{deformation} of the algebra $Q(V, R)$, and if the epimorphism $Q(V, R) \to {\rm gr \,} Q(V, P)$ is an isomorphism, then $Q(V, P)$ is a \emph{PBW-deformation} of $Q(V, R)$.

We will now prove that when the algebra $\mc{C}(A,f)$ is finite-dimensional over $\k$, then it is a PBW-deformation of a Zhang twist by an automorphism, $\Phi$, of the $2$-Veronese subalgebra of the algebra $A^!$. This result will be useful in determining the finite-dimensional representations of $\mc{C}(A,f)$, and in combination with Theorem \ref{NMFs from reps} will also help determine nontrivial noncommutative matrix factorizations of $f$. Additionally, it uncovers that the quadratic dual of the algebra $[A^!_{(2)}]^{\Phi}$ is canonically endowed with the structure of a CDG-algebra, \cite[Chapter 4, Section 4]{PP}. We intend to explore these ideas in future work. 

\begin{construction}\label{PBW deformation}
Let $(A, f) \in \tt{Kos\text{-}QHS}$, and assume that $A$ is noetherian of finite global dimension. By Proposition \ref{normal, regular, Koszul case} we know that $f^! \in \bar{A}^!$ is a normal regular element with normalizing automorphism, $\varphi_{f^!}: \bar{A}^! \to \bar{A}^!$. For notational convenience, let $\Phi = \varphi_{f^!}$.  Since $\Phi$ fixes $f^!$, it is clear that $\Phi(\la f^! \ra) = \la f^! \ra$, hence, abusing notation, we have the induced algebra automorphism $\Phi: \bar{A}^!/\la f^! \ra \to \bar{A}^!/\la f^! \ra$. Recalling that $\bar{A}^!/\la f^! \ra$ is canonically isomorphic to $A^!$, we let $\Phi: A^! \to A^!$ be the algebra automorphism corresponding to $\Phi: \bar{A}^!/\la f^! \ra \to \bar{A}^!/\la f^! \ra$ after identifying $\bar{A}^!/\la f^! \ra \cong A^!$. The restriction of $\Phi: A^! \to A^!$ to the $2$-Veronese subalgebra of $A^!$ will be denoted by $\Phi: A^!_{(2)} \to A^!_{(2)}$. 

Since $A^!$ is Koszul, by \cite[Proposition 3.2.2(i)]{PP} we know that $A^!_{(2)}$ is also a Koszul algebra, in particular,  $A^!_{(2)}$ is a quadratic algebra. We choose a presentation for $A^!_{(2)}$ as follows. Let $\{g_i : 1 \leq i \leq m\}$ be a $\k$-linear basis for the vector space $A^!_2$. Then the kernel of the canonical projection $T(A^!_2) \to A^!_{(2)}$ is generated as an ideal by finitely many elements of the form \[r_k = \sum_{1 \leq i, j \leq m} \l_{i, j, k} \, g_i \tsr g_j, \quad \l_{i, j, k} \in \k, \quad 1 \leq k \leq p \quad \text{for some } p \in \N.\] 
Since $\bar{A}^!/\la f^! \ra \cong A^!$, we may consider $\{g_i : 1 \leq i \leq m\} \cup \{f^!\}$ as a $\k$-linear basis of $\bar{A}^!_2$. For any $1 \leq k \leq p$, the equation $\sum_{i, j} \l_{i, j, k} \, g_ig_j = 0$ in $A^!$ can be lifted to $\bar{A}^!$ yielding the equation \[\sum_{i, j} \l_{i, j, k} \, g_ig_j = b_k f^! \quad \text{for some unique } b_k \in \bar{A}^!_2.\] 
The uniqueness of such $b_k$ follows from the regularity of $f^!$. Then we may write \[b_k = \sum_{i = 1}^m \phi_{i, k} g_i + \theta_k f^! \quad \text{for some unique } \phi_{i, k}, \, \theta_k \in \k.\] Therefore in $\bar{A}^!$ we have the equations \[\sum_{i, j} \l_{i, j, k} \, g_i g_j-\sum_i \phi_{i, k} g_i f^! - \theta_k (f^!)^2 = 0 \quad \text{for }  1 \leq k \leq p.\] 
Right multiplying by $(f^!)^{-2}$ in the ring $\mc{C}(A,f)$ we have \[\sum_{i, j} \l_{i, j, k} \, g_i g_j (f^!)^{-2}-\sum_i \phi_{i, k} g_i (f^!)^{-1} - \theta_k = 0 \quad \text{for }1 \leq k \leq p.\]
Finally, noting that $g_j (f^!)^{-1} = (f^!)^{-1} \Phi(g_j)$, in $\mc{C}(A,f)$ we have \[\sum_{i, j} \l_{i, j, k} \, g_i (f^!)^{-1} \Phi(g_j) (f^!)^{-1}-\sum_i \phi_{i, k} g_i (f^!)^{-1} - \theta_k = 0 \quad \text{for } 1 \leq k \leq p.\]

Let $\{G_i: 1 \leq i \leq m\}$ be a set of formal indeterminates and let $W$ denote the $\k$-vector space spanned by $\{G_i: 1 \leq i \leq m\}$. Then $W$ is isomorphic to $A^!_2$ via the map $G_i \mapsto g_i$ and by abuse of notation we let $\Phi: W \to W$ be the $\k$-linear isomorphism corresponding to $\Phi: A^!_2 \to A^!_2$ under this identification $W \cong A^!_2$. For each $1 \leq k \leq p$ define an element of $T(W)$ by \[R_k = \sum_{i, j} \l_{i, j, k} \, G_i \tsr \Phi(G_j) - \sum_i \phi_{i, k} G_i - \theta_k.\] 
Let $P$ be the subspace of $\k \oplus W \oplus (W \tsr W)$ spanned by the set $\{R_k : 1 \leq k \leq p \}$ and define a nonhomogeneous quadratic algebra by \[Q(W, P) = T(W)/\la P \ra.\]

The above analysis shows that we can define an algebra homomorphism via \[\delta: Q(W, P) \to {\mathcal C}(A, f), \quad  \delta(G_i) = g_i (f^!)^{-1} \quad \text{for all } 1 \leq i \leq m.\] 
Furthermore, since the $g_i (f^!)^{-1}$ generate $\mc{C}(A,f)$ as an algebra, it is clear that $\delta$ is surjective. 

In summary, we have associated to any $(A, f) \in \tt{Kos\text{-}QHS}$, with $A$ a noetherian algebra of finite global dimension, a nonhomogeneous quadratic algebra $Q(W, P)$ and a surjective algebra homomorphism $\delta: Q(W, P) \to {\mathcal C}(A, f)$. This completes the construction.
\end{construction}

Now we have the main result of this section.

\begin{theorem}\label{Cl(A, f) as PBW deformation}
Let $(A, f) \in \tt{Kos\text{-}QHS}$, and assume that $A$ is noetherian of finite global dimension. Adopting the notation of Construction \ref{PBW deformation}, the following statements hold.
\begin{enumerate}
\item[(1)] The algebra homomorphism $\delta: Q(W,P)  \to {\mathcal C}(A, f)$ is an isomorphism.
\item[(2)] The algebra $Q(W, P)$ is a PBW-deformation of the algebra $[A^!_{(2)}]^{\Phi}$. 
\end{enumerate}
\end{theorem}

\begin{proof}


First, by Proposition \ref{Cl(A,f) finite-dimensional}, the algebra $\mc{C}(A,f)$ is finite-dimensional and $\dim_{\k} {\mathcal C}(A, f) = \dim_{k} A^!_{(2)}$.   

As in Construction \ref{PBW deformation}, let $R \subset W \tsr W$ be the subspace spanned by the set \[\Big\{\sum_{i, j} \l_{i, j, k} G_i \tsr \Phi(G_j) : 1 \leq k \leq p\Big\}.\]
The quadratic algebra associated to the pair $(W, P)$ is the algebra \[Q(W, R) = T(W)/\la R \ra,\] and this algebra is clearly isomorphic to the Zhang twist algebra $[A^!_{(2)}]^{\Phi}$. 
Therefore we have \[\dim_{\k} Q(W, R) = \dim_{\k} [A^!_{(2)}]^{\Phi} = \dim_{\k} A^!_{(2)} =  \dim_{\k} {\mathcal C}(A, f),\] 
where the second equality is due to the fact that Zhang twists preserve vector space dimension.

As discussed above, there is a natural filtration on the algebra $Q(W, P)$ induced by the usual filtration on the tensor algebra $T(W)$. Let ${\rm gr \,} Q(W, P))$ denote the associated graded algebra of $Q(W, P)$ with respect to this filtration. As usual with filtered algebras, there is a vector space isomorphism ${\rm gr \,}Q(W, P) \cong Q(W, P)$. 

Recalling the natural epimorphism $Q(W, R) \to {\rm gr \,} Q(W, P))$, we have \[\dim_{\k} Q(W, P) \leq \dim_{\k} Q(W, R),\] 
and it follows that $\dim_{\k} Q(W, P) \leq  \dim_{\k} {\mathcal C}(A, f)$. The algebra morphism \[\delta: Q(W, P) \to {\mathcal C}(A, f)\] 
is surjective, so we also have 
$\dim_{\k} Q(W, P) \geq  \dim_{\k} {\mathcal C}(A, f)$. Therefore \[\dim_{\k} Q(W, P) =  \dim_{\k} {\mathcal C}(A, f),\] and it follows that $\delta: Q(W, P)\to {\mathcal C}(A, f)$ is an isomorphism. 

Noting that we have proved that $\dim_{\k} Q(W, R) = \dim_{\k} Q(W, P)$, it follows that the epimorphism $Q(W, R) \to {\rm gr \,} Q(W, P)$ is an isomorphism, and hence, by definition, $Q(W, P)$ is a PBW-deformation of the algebra $Q(W, R) \cong [A^!_{(2)}]^{\Phi}$.

\end{proof}

\section{Examples}

In this final section we give some examples illustrating our main results. We assume that $\k$ is an algebraically closed field of characteristic zero. We will write matrix factorizations of $f \in A$ as sequences of matrices over $A$. Such sequences are deternined by choosing bases as in the proof of Theorem \ref{F functor valued in NMF}, and then maps of graded free left $A$-modules are given by matrices acting on the right on elements of free modules written as row vectors. Note that with this choice, some care must be taken: if $\phi$ and $\psi$ are endomorphisms of a graded free left $A$-module and represented by matrices $\Phi$ and $\Psi$, respectively, then the composite mapping $\phi \, \psi$ is represented by the matrix $\Psi \, \Phi$. 

\begin{example}\label{quantum plane}
We consider the \emph{quantum plane}. Fix $q \in \k$, $q \ne 0$. Let \[A = \k_{q}[x,y] = \frac{\k\la x, y \ra}{\la yx-qxy \ra}.\]
It is well known that $A$ is a Koszul Artin-Schelter regular algebra of global dimension $2$. 
We now consider two examples of quadratic normal regular elements of $A$, namely $f = x^2$ and $f = xy$.

\noindent {\bf Case 1.} Let $f = x^2 \in A$. Then $f$ is normal with normalizing automorphism \[\varphi_{f}(x) = x, \quad \varphi_{f}(y) = q^{-2} y.\]
The algebra $\bar{A}^!$ is presented as \[\bar{A}^! = \frac{\k \la x, y \ra}{\la y^2, yx+q^{-1}xy \ra},\] and $f^! = x^2 \in \bar{A}^!$. So we may write the quadratic dual of  $A$ as 
\[A^! = \frac{\k \la x, y \ra}{\la x^2, y^2, yx+q^{-1}xy \ra}.\] Therefore we have \[A^!_{(2)} = \k[xy] \cong \frac{\k[g]}{(g^2)}.\]

Following Construction \ref{PBW deformation} we need to lift the equation $(xy)^2 = 0$ in $A^!$ to an equation in $\bar{A}^!$, \[xyxy = bf^! = (\phi xy + \theta x^2)x^2\] for some unique $\phi, \theta \in \k$. It is easy to see that we must have $\phi = \theta = 0$. Using Proposition \ref{normal, regular, Koszul case}, or direct computation, we find that the normalizing automorphism of $f^! \in \bar{A}^!$ is given by \[\varphi_{f^!}(x) = x, \quad \varphi_{f^!}(y) = q^2y.\] So we have $\varphi_{f^!}(xy) = q^2xy$ and it follows that the space $P$ from Construction \ref{PBW deformation} is spanned by $G \tsr \varphi_{f^!}(G) = q^2(G \tsr G)$. 
Therefore we conclude by Theorem \ref{Cl(A, f) as PBW deformation} (1) that ${\mc C}(A, x^2) \cong \k[G]/(G^2)$. From this it is clear that ${\mc C}(A, x^2)$ is not semisimple, so by \cite[Theorem 5.5]{Mori-Ueyama-2}, the algebra $\bar{A}$ is not a graded isolated singularity and does not have finite Cohen-Macaulay representation type. However, the algebra ${\mc C}(A, x^2)$ does have finite representation type, with precisely two indecomposable (right) modules, namely the trivial module $\k$ and the regular module ${\mc C}(A, x^2)$. Using Proposition \ref{hat N to factorization} we now determine the matrix factorizations corresponding to these indecomposable modules. 

If $N = \k$ is the trivial right ${\mc C}(A, x^2)$-module, then the graded pieces of the right $\bar{A}^!$-module \[\widehat{N} = \bigoplus_{i \in \Z} \k \tsr_{{\mc C}(A,x^2)} {\mc Cl}(A, x^2)_i\] are all one-dimensional. Following the proof of Theorem \ref{F functor valued in NMF} we find that the left $A$-linear maps $\phi_i^{\widehat{N}}: {\mc T}(\widehat{N})_i \to {\mc T}(\widehat{N})_{i+1}$ are given by right multiplication by $x \in A$.

If $N = {\mc C}(A, x^2)$ is the regular right ${\mc C}(A, x^2)$-module, then the graded pieces of the right $\bar{A}^!$-module \[\widehat{N} = \bigoplus_{i \in \Z} {\mc C}(A, x^2) \tsr_{{\mc C}(A,x^2)} {\mc Cl}(A, x^2)_i\] are all two-dimensional. Following the proof of Theorem \ref{F functor valued in NMF}, the left $A$-linear maps $\phi_i^{\widehat{N}}: {\mc T}(\widehat{N})_i \to {\mc T}(\widehat{N})_{i+1}$ are given by right multiplication by the following $2 \times 2$ matrices over $A$ acting on $1 \times 2$ row vectors of ${\mc T}(\widehat{N})_i$. We have
\[\phi_i^{\widehat{N}} \equiv \begin{bmatrix} x & q^i y \\ 0 & -q x \end{bmatrix} \quad \text{for } i \equiv 0(2), \quad \phi_i^{\widehat{N}} \equiv \begin{bmatrix} x & q^{i-1} y \\ 0 & -q^{-1} x \end{bmatrix} \quad \text{for } i \equiv 1(2). \]
Using these matrices the reader can check that indeed $\phi_{i+1}^{\widehat{N}} \phi_{i}^{\widehat{N}}$ is given by right multiplication by $x^2 \in A$. For example, if $i \equiv 0(2)$, then \[\phi_{i+1}^{\widehat{N}} \phi_i^{\widehat{N}} \equiv \begin{bmatrix} x & q^i y \\ 0 & -q x \end{bmatrix} \begin{bmatrix} x & q^i y \\ 0 & -q^{-1} x \end{bmatrix} = \begin{bmatrix} x^2 & -q^{i-1}(yx-qxy) \\ 0 & x^2 \end{bmatrix} = \begin{bmatrix} x^2 & 0 \\ 0 & x^2 \end{bmatrix}.\]

\noindent {\bf Case 2.} Next we consider $f = xy \in A$. Then $f$ is normal with normalizing automorphism \[\varphi_{f}(x) = qx, \quad \varphi_{f}(y) = q^{-1}y.\]
The algebra $\bar{A}^!$ is presented as \[\bar{A}^! = \frac{\k \la x, y \ra}{\la x^2, y^2 \ra},\] and $f^! = yx + q^{-1}xy \in \bar{A}^!$. So we may write the quadratic dual of  $A$ as 
\[A^! = \frac{\k \la x, y \ra}{\la x^2, y^2, yx+q^{-1}xy \ra}.\] Therefore we have \[A^!_{(2)} = \k[xy] \cong \frac{\k[g]}{(g^2)}.\]

Following Construction \ref{PBW deformation} we need to lift the equation $(xy)^2 = 0$ in $A^!$ to an equation in $\bar{A}^!$, \[xyxy = bf^! = (\phi xy + \theta(yx+q^{-1}xy))(yx+q^{-1}xy)\] for some unique $\phi, \theta \in \k$. It is easy to see that we must have $\phi = q$,  $\theta = 0$. Using Proposition \ref{normal, regular, Koszul case}, or direct computation, we find that the normalizing automorphism of $f^! \in \bar{A}^!$ is given by \[\varphi_{f^!}(x) = q^{-1}x, \quad \varphi_{f^!}(y) = qy.\] So we have $\varphi_{f^!}(xy) = xy$ and it follows that the space $P$ from Construction \ref{PBW deformation} is spanned by $G \tsr G - qG$. 
Therefore, by Theorem \ref{Cl(A, f) as PBW deformation} (1), we have an isomorphism of algebras, ${\mc C}(A, xy) \cong \k[G]/(G^2-G)$. This makes it clear that ${\mc C}(A, xy)$ is semisimple, so, by \cite[Theorem 5.5]{Mori-Ueyama-2}, the algebra $\bar{A}$ is a graded isolated singularity and has finite Cohen-Macaulay representation type. Moreover, ${\mc C}(A, xy)$ has, up to isomorphism, precisely two one-dimensional simple right modules. Therefore, unsurprisingly, we find two rank 1 noncommutative matrix factorizations of $xy \in A$. 

\end{example}

\begin{example}\label{Sklyanin (1,1,c)}
We consider some three-dimensional \emph{Sklyanin algebras}. Fix $c \in \k$, $c \ne 0$, $c^3 \ne 1$. Let \[A = \frac{\k\la x, y, z \ra}{\la yz+zy+cx^2, zx+xz+cy^2, xy+yx+cz^2 \ra}.\]
This family of algebras was studied in \cite[Section 6]{CG1}, a good reference for some of the unverified details below. It is well known that $A$ is a Koszul Artin-Schelter regular algebra of global dimension $3$. When $c^3 \ne -8$, the point scheme of $A$ is an elliptic curve in $\P^2$. If $c^3 = -8$, then the point scheme of $A$ is the union of three distinct lines in $\P^2$. We treat the cases $c \ne -2$ and $c = -2$ separately in what follows.

Let $f = x^2+y^2+z^2 \in A$. Then it is easy to check that $f$ is a central element. The algebra $\bar{A}^!$ is presented as \[\bar{A}^! = \frac{\k \la x, y, z \ra}{\la [x, y], [x,z], [y,z], (y-x)(cz+y+x), (y-z)(cx+y+z) \ra},\]  where $[r, s] = rs-sr$, and $f^! = -cxz + y^2 \in \bar{A}^!$. 

A \emph{point module} for $\bar{A}^!$ is a $\Z$-graded right $\bar{A}^!$-module $M = \bigoplus_{i \in \Z} M_i$ such that $\dim_{\k} M_i = 1$ for all $i \in \Z$. Up to isomorphism, such a module is determined by the action of $x, y, z \in \bar{A}^!_1$ on $M_i$ for all $i \in \Z$. We describe this action by giving a point $[\l: \mu: \nu] \in \P^2$, and specify that $x, y, z$ act on a fixed basis of $M_i$ by the scalars $\l, \mu, \nu \in \k$, respectively.

\noindent {\bf Case 1.} Assume that $c \ne -2$. Then the commutative algebra $\bar{A}^!$ has precisely four point modules. They are represented by the following distinct points of $\P^2$, \[[1:1:1], \quad [-(1+c): 1: 1], \quad [1: -(1+c): 1], \quad [1:1:-(1+c)].\] Let $M_1, M_2, M_3, M_4$ be the corresponding point modules. We note that the module $M_i$ is an object of the category ${\tt P}(\bar{A}^!, f^!)$. Following the proof of Theorem \ref{F functor valued in NMF}, we find four rank $1$ factorizations of $f = x^2+y^2+z^2$ over the algebra $A$. These factorizations are
\begin{align*}
&f = (x+y+z)((-c+1)^{-1}(x+y+z)) \\
&f = (-(1+c)x+y+z)((c^2+c+1)^{-1}(-(1+c)x+y+z)) \\
&f = (x-(1+c)y+z)((c^2+c+1)^{-1}(x-(1+c)y+z)) \\
&f = (x+y-(1+c)z)((c^2+c+1)^{-1}(x+y-(1+c)z)). \\
\end{align*}
\noindent {\bf Case 2.} Now assume that $c = -2$. In this case, $\bar{A}^!$ has only one point module $M$, represented by $[1:1:1] \in \P^2$. It follows that the four-dimensional commutative algebra ${\mc C}(A, f)$ has only one simple module, so ${\mc C}(A, f)$ is not semisimple. As an application of Theorem \ref{Cl(A, f) as PBW deformation} we will now outline a proof that ${\mc C}(A, f)$ is isomorphic to the algebra $\k[x, y]/(x^2, y^2)$. Following Construction \ref{PBW deformation} and using the notation there, we find that ${\mc C}(A, f)$ is the quotient of the commutative polynomial ring $\k[g_1, g_2, g_3]$ by the ideal $I$ generated by
\begin{align*}
&9g_1g_2 -g_1-g_2+2g_3-1  \\
&9g_1g_3 -g_1+2g_2-g_3-1  \\
&9g_2g_3 +2g_1-g_2-g_3-1  \\
&9g_1^2 -10g_1-4g_2-4g_3+5  \\
&9g_2^2 -4g_1-10g_2-4g_3+5  \\
&9g_3^2 -4g_1-4g_2-10g_3+5.  \\
\end{align*}

Identify ${\mc C}(A, f)$ with the algebra $\k[g_1, g_2, g_3]/I$. Fix a primitive cube root of unity, $\w \in \k$, that is, $\w^3 = 1$ and $\w \ne 1$. Define elements of ${\mc C}(A, f)$ by \[x = \w^2g_1+\w g_2 + g_3, \qquad y = \w g_1+\w^2g_2 + g_3.\] Then one checks that ${\mc C}(A, f)$ is the $\k$-linear span of the linearly independent set $\{1, x, y, xy\}$, and that $x^2 = y^2 = 0$. It follows that ${\mc C}(A, f) \cong \k[x,y]/(x^2, y^2)$. It is well known that this algebra has infinite representation type. Morover, one can check that for all $n \in \N$ and all $\l \in \k$, there is a family $M(n, \l)$ of pairwise non-isomorphic indecomposable modules given by declaring that $x$, $y$ act respectively by \[\begin{bmatrix} 0 & I_n \\ 0 & 0 \end{bmatrix}, \qquad \begin{bmatrix} 0 & J_n(\l) \\ 0 & 0 \end{bmatrix},\] where $I_n$ is the $n \times n$ identity matrix and $J_n(\l)$ is the $n \times n$ Jordan block matrix with $\l$ on the diagonal. We leave it to the interested reader to determine the matrix factorizations of $f$ corresponding to these indecomposables via the functor ${\mc F}$ in Theorem \ref{NMFs from reps}.

\end{example}

\begin{example}\label{S from GKMV}

We consider a four-dimensional Artin-Schelter regular algebra recently studied in detail in \cite{GKMV}. Let \[A = \frac{\k \la x_1, x_2, x_3, x_4 \ra}{\la x_1x_2 - x_3^2, x_2x_1-x_4^2, x_1x_3-x_2x_4, x_4x_1-x_3x_2, x_2x_3-x_3x_1, x_4x_2-x_1x_4 \ra}.\]
As proved in \cite{GKMV}, the algebra $A$ is a Koszul Artin-Schelter regular algebra of global dimension $4$. Furthermore, there are, up to scale, precisely four normal quadratic elements $f \in A$ with distinct normalizing automorphisms, namely 
\[x_1x_4+x_3x_1, \quad x_1x_4-x_3x_1, \quad x_1^2-x_2^2-x_3x_4+x_4x_3, \quad  x_1^2+x_2^2+x_3x_4+x_4x_3.\]

We consider the normal element $f = x_1x_4+x_3x_1 \in A$ in some detail. One checks that the algebra $\bar{A}^!$ has precisely four point modules, which determine, via Theorem \ref{F functor valued in NMF}, four rank $1$ matrix factorizations of $f$. More interestingly, there is also a rank $2$ noncommutative matrix factorization of $f$, which we now describe. 

Define a $\Z$-graded vector space $M = \bigoplus_{i \in \Z} M_i$ with $\dim_{\k} M_i = 2$ for all $i \in \Z$. Then there is a graded right $\bar{A}^!$-module structure on $M$ given as follows. Consider elements of the space $M_i$ to be $2 \times 1$ column vectors and let $x_1, x_2, x_3 , x_4 \in \bar{A}^!$ act on the left, respectively,  by the following matrices \[\begin{bmatrix} 0 & 1 \\ 0 & 0 \end{bmatrix}, \quad \begin{bmatrix} 0 & 0 \\ 1 & 0 \end{bmatrix}, \quad \begin{bmatrix} 0 & 0 \\ 0 & i \end{bmatrix}, \quad \begin{bmatrix} i & 0 \\ 0 & 0 \end{bmatrix},\] 
where $i \in \k$ is a fixed primitive fourth root of unity. Then this makes $M$ into a graded right $\bar{A}^!$-module, and moreover, $M \in {\tt P}(\bar{A}^!, f^!)$. In fact, one can also check that $M$ is a critical module with respect to Gelfand-Kirillov dimension, so $M$ is a fat point module of multiplicity $2$ for $\bar{A}^!$. Following the proof of Theorem \ref{F functor valued in NMF}, we find that the left $A$-linear maps $\phi_i^M: {\mc T}(M)_i \to {\mc T}(M)_{i+1}$ are given by right multiplication by the following $2 \times 2$ matrices over $A$ acting on $1 \times 2$ row vectors of ${\mc T}(M)_i$. We have
\[\phi_0^M \equiv \begin{bmatrix} ix_4 & x_2 \\ x_1 & ix_3 \end{bmatrix}, \ \phi_1^M \equiv \begin{bmatrix} -ix_2 & x_4 \\ x_3 & -ix_1 \end{bmatrix}, \ \phi_2^M \equiv \begin{bmatrix} ix_3 & x_1 \\ x_2 & ix_4 \end{bmatrix}, \ \phi_3^M \equiv \begin{bmatrix} -ix_1 & x_3 \\ x_4 & -ix_2 \end{bmatrix}\]
and $\phi_i^M \equiv \phi_{i(4)}^M$ for all $i \in \Z$. Using these matrices, the reader can check that indeed $\phi_{i+1}^M \phi_{i}^M$ is given by right multiplication by $f \in A$. For example, \[\phi_1^M \phi_0^M \equiv \begin{bmatrix} ix_4 & x_2 \\ x_1 & ix_3 \end{bmatrix} \begin{bmatrix} -ix_2 & x_4 \\ x_3 & -ix_1 \end{bmatrix} = \begin{bmatrix} x_4x_2+x_2x_3 & ix_4^2-ix_2x_1 \\ -ix_1x_2+ix_3^2 & x_1x_4+x_3x_1 \end{bmatrix} = \begin{bmatrix} f & 0 \\ 0 & f \end{bmatrix}.\]

The four point modules and the fat point module $M$ for $\bar{A}^!$ correspond to four one-dimensional simples and one two-dimensional simple for the algebra ${\mc C}(A, f)$. Since ${\mc C}(A, f)$ is eight-dimensional over $\k$, one can conclude via \cite[Proposition 7.2(2)]{Lam1} together with the assumption that $\k$ is algebraically closed, that ${\mc C}(A, f)$ is semisimple and isomorphic to the algebra $\k^{\oplus 4} \oplus {\rm M}_2(\k)$. 

A similar analysis for the other three quadratic normal elements in $f \in A$ yields that, in all cases, ${\mc C}(A, f)$ is semisimple and isomorphic to the algebra $\k^{\oplus 4} \oplus {\rm M}_2(\k)$. Hence, for all quadratic normal elements of $A$, the algebra $\bar{A}$ is a graded isolated singularity, in the sense of \cite[Definition 2.2]{U}. This answers the first query in \cite[Question 5.3.5]{GKMV} in the affirmative. 

\end{example}
\bibliographystyle{amsplain}
\bibliography{bibliog2}

\end{document}